\documentclass[12pt,a4paper,oneside]{amsart}
\usepackage{a4,a4wide}
\usepackage{amsfonts,amsmath,amssymb,amsthm,enumitem}
\usepackage{hyperref}
\usepackage{mathtools}
\usepackage{esint}

\usepackage{color}
\usepackage{ifpdf}
\ifpdf
    \usepackage[pdftex]{graphicx}
    \DeclareGraphicsExtensions{.png,.pdf,.jpg}
\else
   \usepackage[dvips]{graphicx}
   \DeclareGraphicsExtensions{.eps}
\fi
\graphicspath{{.}{figures/}}
\usepackage[latin1]{inputenc}
\numberwithin{equation}{section}

\usepackage{doi}

\usepackage{tikz}
\usetikzlibrary{quotes,angles,positioning}

\newtheorem{theorem}{Theorem}[section]
\newtheorem{lemma}[theorem]{Lemma}

\newtheorem{remark}[theorem]{Remark}
\newtheorem*{theorem*}{Theorem}
\newtheorem*{lemma*}{Lemma}
\newtheorem*{proposition*}{Proposition}
\newtheorem*{corollary*}{Corollary}
\renewcommand\tilde{\widetilde}

\def\R{\mathbb{R}}

\def\Z{\mathbb{Z}}
\def\N{\mathbb{N}}
\def\EE{\mathbb{E}}

\def\P{\mathbb{P}}

\renewcommand{\phi}{\varphi}

\def\1{\mathbf{1}}

\def\XXint#1#2#3{{\setbox0=\hbox{$#1{#2#3}{\int}$ }
\vcenter{\hbox{$#2#3$ }}\kern-.57\wd0}}

\def\eps{\varepsilon}

\def\les{\lesssim}
\def\ges{\gtrsim}

\newcommand{\bra}[1]{\left( #1 \right)}
\newcommand{\sqa}[1]{\left[ #1 \right]}
\newcommand{\cur}[1]{\left\{ #1 \right\}}

\def\leb{\mathsf{Leb}}
 \newcommand{\abs}[1]{\left\vert#1\right\vert}

\newcommand{\id}{\mathsf{id}}

\newcommand{\mres}{\mathbin{\vrule height 1.6ex depth 0pt width
0.13ex\vrule height 0.13ex depth 0pt width 1.3ex}}

\begin{document}
\title[There is no stationary c.~m.\ Poisson Matching in 2d]{There is no stationary cyclically monotone Poisson matching in 2d}

\author{  Martin Huesmann \address[Martin Huesmann]{Universit\"at M\"unster,  Germany} \email{martin.huesmann@uni-muenster.de} \hspace*{0.5cm} Francesco Mattesini \address[Francesco Mattesini]{Universit\"at M\"unster \& MPI Leipzig,  Germany} \email{francesco.mattesini@uni-muenster.de}  \hspace*{0.5cm} Felix Otto \address[Felix Otto]{MPI Leipzig, Germany} \email{Felix.Otto@mis.mpg.de} } 
\thanks{All authors are supported by the Deutsche Forschungsgemeinschaft (DFG, German Research Foundation) through the SPP 2265 {\it Random Geometric Systems}. MH and FM have been funded by the Deutsche Forschungsgemeinschaft (DFG, German Research Foundation) under Germany's Excellence Strategy EXC 2044 -390685587, Mathematics M\"unster: Dynamics--Geometry--Structure .  }

\begin{abstract}
 We show that there is no cyclically monotone stationary
 matching of two independent Poisson processes in dimension $d=2$. The proof combines the harmonic approximation result from \cite{GHO} with local asymptotics for the two-dimensional matching problem for which we give a new self-contained proof using martingale arguments.
\end{abstract}

\date{\today}
\maketitle

\section{Introduction}

Consider two locally finite\footnote{and thus countable}
point sets $\{X\},\{Y\}\subset\mathbb{R}^d$ in $d$-dimensional
space. We are interested in their matching, which we think of as a bijection $T$ from
$\{X\}$ onto $\{Y\}$. More specifically, we are interested in their matching by
cyclically monotone maps $T$,
which means that for any finite subset\footnote{which we momentarily enumerate} 
$\{X_n\}_{n=1}^{N}$ we have
\begin{align}\label{fw01}
\sum_{n=1}^N T(X_n)\cdot(X_n-X_{n-1})\ge 0\quad\mbox{with}\quad X_0:=X_N.
\end{align}
It is elementary to see that (\ref{fw01}) is equivalent to local optimality, meaning
\begin{align}\label{eq:loc opt}
\sum_X(|T(X)-X|^2-|\tilde T(X)-X|^2)\le 0
\end{align}
for any other bijection $\tilde T$ that differs from $T$ only on a finite number
of points\footnote{so that the sum is effectively finite}. 
This makes a connection to the optimal transportation between the
measures
\begin{equation}\label{eq:defmunu}
\mu=\sum_{X}\delta_X \quad \mbox{and} \quad \nu=\sum_{Y}\delta_Y
\end{equation}
related via $T\#\mu=\nu$, 
which we shall explore in this paper. Note however
that because of the (typically) infinite number of points, we cannot view $T$ as a minimizer. 

\medskip

Before proceeding to the random setting, we make two simple observations
that show that the set of $(\{X\},\{Y\},T)$ is rich:
For $d=1$, (\ref{fw01}) is easily seen to be equivalent to plain monotonicity. 
Every single matching $T(X_0)=Y_0$
can obviously be extended in a unique way to a monotone bijection $T$ of $\{X\}$ and $\{Y\}$, 
so that for $d=1$, the set of monotone bijections $T$ has the same magnitude as $\{X\}$ itself.
Returning to general $d$, we note that for any 
monotone bijection $T$ of $\{X\}$ and $\{Y\}$, and for any two shift vectors 
$\bar x,\bar y\in\mathbb{R}^d$,
the map $x\mapsto T(x-\bar x)+\bar y$ is a monotone bijection of the shifted point sets
$\{\bar x+X\}$ and $\{\bar y+Y\}$.

\medskip

We are interested in the situation when the sets $\{X\}$, $\{Y\}$ and their cyclically monotone
bijection $T$ are random. More precisely, we consider the case when $\{X\}$ and $\{Y\}$
are independent Poisson point processes of unit intensity. We assume that the $\sigma$-algebra
for $(\{X\},\{Y\},T)$ is rich enough so that the following elementary observables
are measurable, namely the number $N_{U,V}$ of matched pairs $(X,Y)\in U\times V$ for
any two Lebesgue-measurable sets $U,V\subset\mathbb{R}^d$ (with $U$ or $V$ having finite Lebesgue 
measure\footnote{so that the following number is finite}):
\begin{align*}
N_{U,V}:=\#\{(X,Y)\in U\times V|Y=T(X)\}\in\{0,1,\cdots\}.
\end{align*}
We now come to the crucial assumption on the ensemble: 
In view of the above remark, the additive group 
$\mathbb{Z}^d\ni\bar x$ acts on $(\{X\},\{Y\},T)$ via
\begin{align*}
(\{X\},\{Y\},T)\mapsto(\{\bar x+X\},\{\bar x+Y\},T(\cdot-\bar x)+\bar x).
\end{align*}
We assume that this action is stationary and ergodic\footnote{Recall that this is
automatic for the two first components $(\{X\},\{Y\})$ by the properties of the
Poisson point process so that this amounts to a property of $T$}. On the one hand, stationarity is
a structural assumption; we shall only use it in following form: For any shift vector $\bar x$,
the random natural numbers $N_{\bar x+U,\bar x+V}$ and $N_{U,V}$ have the same distribution. 
On the other hand, ergodicity is a qualitative assumption\footnote{paraphrased by saying
that spatial averages are ensemble averages}; we will only use it in form of
the following application of Birkhoff's ergodic theorem:
\begin{align*}
\lim_{R\uparrow\infty}\frac{1}{R^d}\sum_{\bar x\in\mathbb{Z}^d\cap[0,R)^d}N_{\bar x+U,\bar x+V}
=\mathbb{E}N_{U,V}\quad\mbox{almost surely}.
\end{align*}

\medskip

\begin{theorem} \label{thm:main}
For $d\le 2$, there exists no stationary and ergodic ensemble
of $(\{X\},$ $\{Y\},T)$, where $\{X\}$, $\{Y\}$ are independent Poisson point processes
and $T$ is a cyclically monotone bijection of $\{X\}$ and $\{Y\}$.
\end{theorem}

\medskip

Our interest in this problem is motivated on the one hand by work on geometric properties of matchings by Holroyd \cite{Ho11} and Holroyd et al. \cite{HoPePeSc09, HoJaWa20}, and on the other hand by work on optimally coupling random measures by the first author and Sturm \cite{HuSt13} and the first author \cite{Hu16}. 
In \cite{HoJaWa20}, Holroyd, Janson, and W\"astlund analyze (stationary) matchings satisfying the local optimality condition \eqref{eq:loc opt} with the exponent 2 replaced by $\gamma\in[-\infty,\infty].$ They call matchings satisfying this condition $\gamma$-minimal and derive a precise description of the geometry of these matchings in dimension $d=1$. In dimension $d>1$ much less is known. In particular, in the critical dimension $d=2$  they could only show existence of stationary $\gamma$-minimal matchings for $\gamma<1$. 
The cases $\gamma\geq 1$ were left open, but see \cite{HoPePeSc09} and \cite{Ho11} for several  open questions for $d=2$ and in particular $\gamma=1$.
On the other hand, the first author and Sturm \cite{HuSt13, Hu16} develop an optimal transport approach to this (and related) problems. They identify the point sets $\{X\},\{Y\}$ with the counting measures \eqref{eq:defmunu} and seek a stationary coupling $Q$ between $\mu$ and $\nu$ minimizing the cost 
\begin{align*}
 \EE\int_{B_1\times \R^d} |x-y|^\gamma dQ.
\end{align*}
(Note that any (stationary) bijection $T:\{X\}\to\{Y\}$ induces a (stationary) coupling between $\mu$ and $\nu$ by setting $Q=(\id,T)_\#\mu$.) 
If this cost functional is finite there exists a stationary coupling which is necessarily locally optimal in the sense of \eqref{eq:loc opt} with the exponent 2 replaced by $\gamma$. In dimension 2 for $\mu$ and $\nu$ being two independent Poisson processes, this functional is finite if and only if $\gamma<1$, which is in line with the results of \cite{HoJaWa20}.

\medskip

In view of these results, it is natural to conjecture that Theorem \ref{thm:main} also holds for $\gamma$-minimal matchings with $\gamma\geq 1$. However, our proof crucially relies on the harmonic approximation result of \cite{GHO} which so far is only available for $\gamma=2$.

\medskip

Before we explain the main steps of the proof of Theorem \ref{thm:main} in Section \ref{sec:mainsteps} we would like to give a few remarks on extensions and variants of Theorem \ref{thm:main}.

\begin{remark}
Theorem \ref{thm:main} remains true if we replace the bijection $T$ by the a priori more general object of a stationary coupling $Q$.
This can be seen by either using that matchings are extremal elements in the set of all couplings of point sets or by directly writing the proof in terms of couplings which essentially only requires notational changes.
\end{remark}

\begin{remark}
Very well studied siblings of stationary matchings are stationary allocations of a point process $\{X\}$, i.e.\ a stationary map $T:\R^d\to\{X\}$ such that $\mathsf
{Leb}(T^{-1}(X))$ equals $\EE[\#\{X\in(0,1)^d\}]^{-1}$.
There are several constructions of such an allocation, for instance by using the stable marriage algorithm in \cite{HoHoPe06}, the flow lines of the gravitational force field exerted by $\{X\}$ in \cite{ChPePeRo10}, an adaptation of the AKT scheme in \cite{MaTi16}, or by optimal transport methods in \cite{HuSt13}.

By essentially the same proof as for Theorem \ref{thm:main}, one can show that in $d=2$ there is no cyclical montone stationary allocation to a Poisson process. The only place where we need to change something in the proof is the $L^\infty$ estimate Lemma \ref{lem:Linfty}.
\end{remark}

\begin{remark}
We do not use many particular features of the Poisson measures $\mu$ and $\nu$ in the proof of Theorem \ref{thm:main} since ergodicity and stationarity allow us to argue on a pathwise level via the harmonic approximation result (cf.\ Section \ref{sec:mainsteps}).

We use two properties of the Poisson measure. The first property is concentration around the mean. The second property is more involved. Denote by $W_p$ the $L^p$ Wasserstein distance. We use that 
$\frac1{R^d}W_2(\mu \mres B_R,\frac{\mu(B_R)}{|B_R|}\mathsf{Leb})$ diverges at the same rate for $R\to\infty$ as $\frac{1}{R^d} W_{2-\eps}(\mu \mres B_R,\frac{\mu(B_R)}{|B_R|}\mathsf{Leb})$
for some $\eps>0$ (here we use $\eps=1$). 
\end{remark}

As the last remark indicates stationary matchings are closely related to the bipartite matching problem, which is the natural variant of the problem studied in this paper with only a finite number of points $\{X_1,\ldots,X_n\},\{Y_1,\ldots,Y_n\}$. Note that then the local optimality condition \eqref{eq:loc opt} turns into a global optimality condition\footnote{which in measure theoretic formulation used in optimal transport is nothing but the cyclical mononotonicity of the support of the corresponding coupling.}. This (finite) bipartite matching problem has been the subject of intense research in the last 30 years, see e.g.\ \cite{AKT84} for the first proof of the rate of convergence in the case of iid points in dimension $d=2$, \cite{Ta94} for sharp integrability properties for matchings in $d\geq 3$, \cite{CaLuPaSi14} for a new appraoch based on the linearization of the Monge-Amp\`ere equation, and \cite{BobLe} for a thorough analysis of the bipartite matching problem in $d=1$.

\medskip

We rely on this connection in two ways. On the one hand we use the asymptotics of the cost of the bipartite matching which are known since \cite{AKT84}. However, we need a local version for which we give a self-contained proof using martingale techniques (see Section \ref{sec:bounds}) which is new and interesting on its own. On the other hand we exploit a large scale regularity result for optimal couplings, the harmonic approximation result, developed in \cite{GHO}. This regularity result was inspired by the PDE approach proposed by \cite{CaLuPaSi14}, see also \cite{AmStTr16, AGS19, Le17, GolTre} for remarkable results for the bipartite matching problem using this approach.

\subsection{Main steps in the proof of Theorem \ref{thm:main}}\label{sec:mainsteps}

In the following we will describe the main steps in the proof of Theorem \ref{thm:main}. For the detailed proofs we refer to Section \ref{sec:proofs}.

\medskip

We argue by contradiction. We assume that there is a locally optimal stationary matching $T$ between $\cur{X}$ and $\cur{Y}$. On the one hand, we will show that
\begin{equation}\label{eq:upperbound}
 \frac1{R^d}\sum_{X \in B_R\;\mbox{or}\;T(X) \in B_R} |T \bra{X} - X| \leq  o(\ln^{\frac12}R).
\end{equation}
On the other hand, it is known (we will prove the local version needed for our purpose) that any bipartite matching satisfies
\begin{equation}\label{eq:lowerbound}
 \frac1{R^d}\sum_{X \in B_R\;\mbox{or}\;T(X) \in B_R} |X-T \bra{X}| \geq O(\ln^{\frac12}R) 
\end{equation}
leading to the desired contradiction. 

\medskip
Let us now describe the different steps leading to \eqref{eq:upperbound} and \eqref{eq:lowerbound} in more detail. 
Our starting point is the observation  that
%that the assumption of the existence of $q$ implies
%\begin{equation}\label{eq:coupling finite}
 %q(\{(x,y)\in [0,1]^2\times\R^2:|x-y|<\infty\})= \mu([0,1]^2) \quad a.s..
%\end{equation}
% We will also use the concentration property of Poisson random variables, i.e.\ if $N$ is a Poisson random variable with parameter $\lambda$ then we have
% $$ \P[|N-\lambda|>r]\leq 2\exp\bra{- \frac{r^2}{2(r+\lambda)}}.$$
by stationarity and ergodicity the following $L^0$-estimate on the displacement $T \bra{X} - X$ holds
%It shows that an arbitrarily high volume fraction will be mapped no further than distance $M$.
 \begin{equation}\label{eq:L0Sketch}
\# \cur{X \in (-R,R)^d \ \colon \ \abs{T \bra{X} - X} \gg 1} \leq o (R^d),
 \end{equation}
see Lemma \ref{lem:L0} for a precise statement.
% This means that a very high volume fraction of $\left[-\frac{R}{2},\frac{R}{2}\right]^2$ will not be mapped further than $M$. 
%Fix $\eps$ and $R\geq r_*$ so that \eqref{eq:L0Sketch} holds.
Since $T$ is locally optimal, its support is in particular monotone, which means that for any $X, X' \in \cur{X}$ we have
$$(T \bra{X'} - T \bra{X})\cdot (X'-X) \geq 0.$$
By \eqref{eq:L0Sketch}, we know that most of the points in $(-R,R)^d$ are not transported by a large distance. Combining this with monotonicity allows us to also shield the remaining points  from being transported by a distance of order $R$, so that
 \begin{equation}\label{eq:LinftySubSketch}
   \abs{T \bra{X} - X} \leq  o(R)\; \mbox{provided that} \; X \in (-R,R)^d,
 \end{equation}
see Lemma \ref{lem:Linfty} for a precise statement. By concentration properties of the Poisson process we may assume that $\frac{\# \cur{X \in B_R}}{\abs{B_R}}\in \sqa{\frac12,2}$ for $R\gg 1$. 
Summing \eqref{eq:LinftySubSketch} over $B_R$ we obtain 
%\begin{equation*}
% \frac{1}{R^d} \sum_{X \in B_R} |T \bra{X} - X|^2 \les \eps R^2
%\end{equation*}
%so that by arbitrariness of $\eps$ we can restate this by
\begin{equation}\label{eq:oR2}
 \frac{1}{R^d} \sum_{X \in B_R} |T \bra{X} - X|^2 \leq o(R^2).
\end{equation}
Now comes the key step of the proof. We want to exploit regularity of $T$    
to upgrade \eqref{eq:oR2} to an $O(\ln R)$ bound. The  tool that allows us to do this is the harmonic approximation result \cite[Theorem 1.4]{GHO} which quantifies the closeness of the displacement $T \bra{X} - X$ %under the matching $T$ 
to a harmonic gradient field $\nabla \phi$, taking as input only the local energy
\begin{equation} \label{eq:energy}
  E(R) :=\frac{1}{R^d} \sum_{X \in B_R \ \mbox{or} \ T \bra{X} \in B_R} |T \bra{X} - X|^2
\end{equation}
and the distance of $\mu \mres (-R,R)^d$ and $\nu \mres (-R,R)^d$ to the Lebesgue measure on $B_R$
\begin{equation}\label{eq:data}
D(R) := \frac{1}{R^d}W^2_{(-R,R)^d}(\mu,n_\mu) + \frac{R^2}{n_\mu}(n_\mu-1)^2 +\frac{1}{R^d}W^2_{(-R,R)^d}(\nu,n_\nu) + \frac{R^2}{n_\nu}(n_\nu-1)^2,
\end{equation}
where $n_\mu = \frac{\# \cur{X \in B_R}}{\abs{B_R}}$ and $n_\nu = \frac{\# \cur{Y \in B_R}}{\abs{B_R}}$, and $W_B \bra{\mu, n} = W_2 \bra{\mu \mres B, n \leb \mres B}$. By \eqref{eq:oR2} (together with its counter part arising from exchanging the roles of $\cur{X}$ and $\cur{Y}$) we have $E(R)\leq o(R^2)$. 
%Lemma \ref{Lem2} together with concentration properties of the Poisson process imply the 
By the well-known bound for the matching problem in $d=2$ and by concentration properties of the Poisson process we have $D(R)\leq O(\ln R)$, see Lemma \ref{lem:upperboundpathwise}. Iteratively exploiting the harmonic approximation result on an increasing sequence of scales we obtain that the local energy $E$ inherits the asymptotic of $D$:
%{\color{blue} MH: How explicit do we want to be with the definition of $r_*$?}
\begin{equation}\label{eq:L2statementSketch}
\frac{1}{R^d}\sum_{X \in B_R \; \mbox{or} \; T \bra{X} \in B_R} |T\bra{X}  - X|^2  \leq  O(\ln R),
\end{equation}
see Lemma \ref{lem:harmonicapprox}. 
Combining this with the $L^0$-estimate yields 
$$ \frac{1}{R^d} \sum_{X\in B_R \;\mbox{or}\; T\bra{X} \in B_R}|T \bra{X} - X|  \leq  o ( \ln^{\frac12} R),$$
see Lemma \ref{lem:upperbound}.

\medskip

It remains to establish the lower bound \eqref{eq:lowerbound}, which is essentially known. However, for our purpose we need the local version \eqref{eq:lowerbound}. Our proof is very similar to the proof of the lower bound in the seminal paper \cite{AKT84}. Both approaches construct candidates for the dual problem based on dyadic partitions of the cube. However, instead of using a quantitative embedding result into a Gaussian process as in \cite{AKT84} we use a natural martingale structure together with a concentration argument. More precisely, we show that there exists $\zeta$ with ${\rm supp} \zeta \in \bra{0,R}^2$ and $\abs{\nabla \zeta} \leq 1$ such that
\[
\frac1{R^d} \sum_{X \in (0,R)^2\;\mbox{or}\;T(X) \in (0,R)^2}  \zeta \bra{T \bra{X}} - \zeta \bra{X}  \geq O(\ln^\frac12 R),
\]
see Lemma \ref{lem:lowerboundpathwise}. Note that this is sufficient for the contradiction stated at the beginning of the subsection, indeed
\[
\begin{split}
O(\ln^\frac12 R) & \leq \frac1{R^d} \sum_{X \in (0,R)^2\;\mbox{or}\;T(X) \in (0,R)^2}  \zeta \bra{T \bra{X}} - \zeta \bra{X} \\
&  \leq \frac1{R^d}\sum_{X \in (0,R)^2\;\mbox{or}\;T(X) \in (0,R)^2} \abs{T \bra{X} - X} \\
& \leq  \frac1{R^d}\sum_{X \in B_{\sqrt{2}R}\;\mbox{or}\;T(X) \in B_{\sqrt{2}R}} \abs{T \bra{X} - X} \\
& \leq o(\ln^\frac12 R).
\end{split}
\]

%The proof is very similar to the proof of the lower bound in the seminal paper \cite{AKT84}. Both approaches construct candidates for the dual problem based on dyadic partitions of the cube. However, instead of using a quantitative embedding result into a Gaussian process as in \cite{AKT84} we use a natural martingale structure together with a concentration argument. 

%\subsection{Notation ??}

%{\color{blue} add somewhere a senetence that $r_*$ might change from line to line or so....}

\section{Proofs}\label{sec:proofs}
\subsection{The ergodic estimate}
\begin{lemma}\label{lem:L0}
 For any $\eps>0$ there exist a deterministic $L$ and a random radius $r_* < \infty$ a.~s. such that for all $R\geq r_*$
 \begin{equation}\label{eq:L0}
 \# \cur{X \in (-R,R)^d \,|\,\abs{T \bra{X} - X} > L} \leq \bra{\eps R}^d. 
 \end{equation}
\end{lemma}

\begin{proof}
Let $Q_R = (-R,R)^d$ and consider the number of points in $Q_R$ which are transported by a distance greater than $L$, namely
%\begin{equation}\label{eq:Lambda}
\[
N_{Q_R, B_L \bra{Q_R}^c}  = \# \cur{ X \in Q_R \,|\,\abs{T \bra{X} - X} > L }.
\]
%\end{equation}
We show that stationarity together with the ergodic theorem implies as $R \rightarrow \infty$ that
\begin{equation}\label{eq:ergconv}
\frac{1}{R^d}  N_{Q_R, B_L \bra{Q_R}^c}\rightarrow \mathbb{E}  N_{Q_1, B_L \bra{Q_1}^c}  \quad \mbox{a.~s.}.
\end{equation}
Then, taking  $L \rightarrow \infty$ we have
\begin{equation}\label{eq:statconv}
\mathbb{E}  N_{Q_1, B_L \bra{Q_1}^c}  \rightarrow 0.
\end{equation}
Note that \eqref{eq:ergconv} together with \eqref{eq:statconv} imply the existence of a random radius $r_*$ and a deterministic $L$ as in the Lemma, indeed for any fixed $\eps>0$ we can choose $L$ large enough so that $\mathbb{E} N_{Q_1, B_L \bra{Q_1}^c} \leq \frac{\eps^d}{2}$, and then choose $r_*$ large enough so that for $R \geq r_*$ $$\abs{\frac{1}{R^d} N_{Q_R, B_L \bra{Q_R}^c} - \mathbb{E} N_{Q_1, B_L \bra{Q_1}^c}} \leq \frac{\eps^d}{2},$$
which turns into the equivalent assertion $\frac{1}{R^d}  N_{Q_R, B_L \bra{Q_R}^c} \leq \eps^d$. 

We now turn to \eqref{eq:ergconv}. For $i \in \Z^d$ note that  
\[
\begin{split}
 N_{Q_R, B_L \bra{Q_R}^c}  &= \sum_{\abs{i}< R} N_{Q_1 + i, B_L \bra{Q_1 + i}^c} \\
& = \sum_{\abs{i}<R} \# \cur{ X \in { \bra{i_1 -1 ,i_1+1 } \times \dots \times \bra{i_d -1 , i_d+1 }}  \,|\,\abs{T \bra{X} - X} > L}.
\end{split}
\] 
Since $\mu$ is stationary and ergodic, by the Birkhoff-von Neumann ergodic theorem \cite[Theorem 9.6]{kallenberg} for any divergent sequence of integer radii
\begin{equation}\label{eq:Birkhconv1}
\frac{1}{R^d} \sum_{\abs{i} < R}  N_{Q_1 + i, B_L \bra{Q_1 + i}^c} \rightarrow \mathbb{E} N_{Q_1, B_L \bra{Q_1}^c}.
\end{equation}
Note that for integer $R$ the term on the left hand side amounts to $\frac{1}{R^d} N_{Q_R, B_L \bra{Q_R}^c}$. Since for every real $R$ there exists an integer $\bar{R}$ such that $\bar{R} \leq R \leq \bar{R}+1$ and the ratio between $\bar{R}$ and $\bar{R}+1$ goes to $1$ as $R \rightarrow \infty$, \eqref{eq:Birkhconv1} holds also for any divergent sequences of real radii.  

Finally we turn to \eqref{eq:statconv}. By dominated convergence it suffices to show that
\begin{align}\label{eq:pointwiseErgodic}
 N_{Q_1, B_L \bra{Q_1}^c} \rightarrow 0 && a.~s..
\end{align}
Indeed, since $ N_{Q_1, B_L \bra{Q_1}^c} \leq  \# \{X \in Q_1 \}$ by dominated convergence \eqref{eq:statconv} holds. Note that $ N_{Q_1, Q_L^c}$ is finite for every realization of $\cur{X}$, then there exists $L$ large enough such that $ N_{Q_1, Q_L^c} =0$, which implies the almost sure convergence \eqref{eq:pointwiseErgodic} since $\EE N_{Q_1, B_L \bra{Q_1}^c}$ is finite by the properties of the Poisson process.  
%Finally, for any $\eps > 0$, since the measure $q = \bra{\mathsf{id},T} \bra{\mu}$ is finite on $\sqa{0,1} \times \R^2$ there exists a ball of radius $R> \tilde{R}$, for some  constant $\tilde{R}>0$, such that $q \bra{\bra{\sqa{0,1}^2 \times \R^2}\setminus \sqa{0,1}^2 \times B_R} \leq \eps$, which implies the pointwise convergence 
%\[
%\lim_{M \rightarrow \infty} \mu \bra{\Lambda^c \bra{1,M}} = 0.
%\]
%Since $\mu \bra{ \Lambda^c \bra{1,M}} \leq \mu \bra{ \sqa{0,1}^2}$, by dominated convergence \eqref{eq:statconv} holds.
\end{proof}

\subsection{The $L^\infty$-estimate}
The proof is very similar to the $L^\infty$ estimate in \cite[Lemma 2.9]{GHO}. However, note that in  \cite[Theorem 1.4]{GHO} a local $L^2$ estimate was turned into a $L^\infty$ estimate whereas in the current setup we want to turn a $L^0$ estimate into a $L^\infty$ estimate. The key property that allows us to do this is the monotonicity of the support of $T$. This translates the partial control given by \eqref{eq:L0} into the claimed $L^\infty$ estimate.

\begin{lemma}\label{lem:Linfty}
For every $\eps >0$ there exists a random radius $r_*<\infty$ a.~s.  such that for every $R \geq r_*$
 \begin{equation}\label{eq:Linfty}
   \abs{T \bra{X} - X} \leq  \eps R \quad \mbox{provided that} \; X \in (-R,R)^d.
 \end{equation}
\end{lemma}
\begin{proof}

\medskip

{\sc Step 1}. Definition of $r_*=r_*(\eps)$ given $0<\eps\ll 1$ 
as the maximum of three $r_*$'s.
First, by Lemma \ref{lem:L0}, there exists a (deterministic) length $L<\infty$ and 
the (random) length $r_*<\infty$ such that for $4R\ge r_*$,
the number density of the Poisson points in $(-2R,2R)^d$ 
transported further than the ``moderate distance'' $L$ is small in the sense of
\begin{align}\label{fw17}
\#\{\,X\in(-2R,2R)^d\,|\,|T(X)-X|>L\,\}\le(\eps 4R)^d.
\end{align}
Second, by Lemma \ref{lem:upperboundpathwise} we may also assume that $r_*$ is so large that for $R\ge r_*$,
the non-dimensionalized transportation distance of $\mu$ to its number density $n$
is small, and that $n\approx 1$, in the sense of
\begin{align}\label{fw11}
W^2_{(-2R,2R)^d}(\mu,n)+\frac{(4R)^{d+2}}{n}(n-1)^2\le
(\eps 4 R)^{d+2}.
\end{align}
Third, w.~l.~o.~g.~we may assume that $r_*$ is so large that 
\begin{align}\label{fw14}
L\le\eps r_*.
\end{align}
We now fix a realization and $R\ge r_*$.

\medskip

{\sc Step 2}. There are enough Poisson points on mesoscopic scales.
We claim that for any cube $Q\subset(-2R,2R)^d$ of ``mesoscopic'' side length 
\begin{align}\label{fw12}
r\gg\eps R
\end{align}
we have\footnote{The notation $A \gtrsim B$ means that there exists a constant $C>0$, which may only depend on $d$, such that $A \geq C B$. Similarly $A \lesssim B$ means that there exists a constant $C>0$, which may only depend on $d$, such that $A \leq C B$.}
\begin{align}\label{fw13}
\#\{\,X\in Q\,\}\gtrsim r^d. 
\end{align}
Indeed, it follows from the definition of $W_{(-2R,2R)}(\mu,n)$ that for any
Lipschitz function $\eta$ with support in $Q$ we have
\begin{align*}
\big|\int\eta d\mu-\int\eta ndy\big|\le({\rm Lip}\eta)\big(\int_Qd\mu+n|Q|\big)^\frac{1}{2}
W_{(-2R,2R)^d}(\mu,n).%\ge n\int_{(-2R,2R)^d}\eta.
\end{align*}
We now specify to an $\eta\le 1$ supported in $Q$, to the effect of $\int\eta d\mu$
$\le\int d\mu$ $=\#\{\,X\in Q\,\}$, so that by Young's inequality
\begin{align}\label{fw25}
\int\eta ndy&\lesssim
\#\{\,X\in Q\,\}+({\rm Lip}\eta)^2W_{(-2R,2R)^d}^2(\mu,n)\nonumber\\
&+({\rm Lip}\eta)\big(n|Q|\big)^\frac{1}{2}
W_{(-2R,2R)^d}(\mu,n).%\ge n\int_{(-2R,2R)^d}\eta.
\end{align}
At the same time, we may ensure $\int_{(-2R,2R)^d}\eta\gtrsim r^d$ and 
${\rm Lip}\eta\lesssim r^{-1}$, 
so that by (\ref{fw11}), which in particular ensures $n\approx 1$, (\ref{fw25})
turns into
\begin{align*}
r^d\lesssim 
\#\{\,X\in Q\,\}+r^{-2}(\eps R)^{d+2}+r^{\frac{d}{2}-1}(\eps R)^{\frac{d}{2}+1}.
\end{align*}
Thanks to assumption (\ref{fw12}) we obtain (\ref{fw13}).

\medskip

{\sc Step 3}. Iteration. There are enough Poisson points of moderate transport distance
on mesoscopic scales. We claim that for any cube
$Q\subset(-2R,2R)^d$ of side-length satisfying (\ref{fw12}) 
we have
\begin{align}\label{fw16}
\mbox{there exists}\;X\in Q\;\mbox{with}\;|T(X)-X|\le L.
\end{align}
We suppose that (\ref{fw16}) were violated for some cube $Q$. 
%In view of (\ref{fw14}),
%all Poisson points in $X$ would be transported a distance larger than $L$.
By (\ref{fw13}), there are $\gtrsim r^d$ of such points. By assumption
(\ref{fw12}), there are thus $\gg(\eps R)^d$ Poisson points in $(-2R,2R)^d$
that get transported by a distance $>L$, which contradicts (\ref{fw17}).

\medskip

{\sc Step 4}. At mesoscopic distance around a given point $X\in(-R,R)^d$, 
there are sufficiently
many Poisson points that are transported only over a moderate distance. 
More precisely, we claim that provided (\ref{fw12}) holds, 
there exist $d+1$ Poisson points $\{X_n\}_{n=1}^{d+1}$ 
that are transported over a moderate distance, i.~e.
\begin{align}\label{fw19}
|T(X_n)-X_n|\le L,
\end{align}
but on the other hand lie in ``sufficiently general'' directions around $X$, meaning that 
\begin{align}\label{fw20}
\begin{array}{l}
\mbox{the convex hull of}\;\{\frac{X_n-X}{|X_n-X|}\}_{n=1}^{d+1}\;\mbox{contains}\;B_\rho
\end{array}
\end{align}
for $\rho\ll 1$, while the distances to $X$ are of order $r$
\begin{align}\label{fw18}
|X_n-X|\sim r.
\end{align}
Indeed, this can be seen as follows: Consider the symmetric 
tetrahedron\footnote{using the $d=3$-language, for $d=2$ it is the equilateral triangle} 
with barycenter at $X$. For each of its $d+1$ vertices, consider a rotationally
symmetric cone with apex at $X$ and with axis passing through the vertex.
Provided the opening angles are $\ll 1$, by continuity,
any selection $e_n$ of unit vectors in this cones still has the property 
that their convex hull contains $B_\rho$ for $\rho\ll 1$.
%Fix $d+1$
%rotationally symmetric cones with apex at $X$, 
%small opening angle, and axes in the directions of the $d+1$ vertices of  
%a symmetric tetrahedron\footnote{using the $d=3$-language, for $d=2$ it is the equilateral
%triangle} with barycenter at $X$. 
Consider the intersection of these cones with
the (dyadic) annulus centered at $X$ of radii $r$ and $2r$. These $d+1$ intersections are contained
in $(-2R,2R)^d$, and each contains a cube of side-length $\sim r$. Hence (\ref{fw16})
applies and we may pick a Poisson point $X_n$ with (\ref{fw19}) in each of these intersections, see Figure \ref{fig:1}.
Condition (\ref{fw18}) is satisfied because of the annulus, and condition (\ref{fw19})
is satisfied because of the cones.

\begin{figure}[h]
\begin{tikzpicture}
  \draw (0,0) circle (1cm);
  \draw (0,0) circle (2cm);
  \coordinate (A) at (0,0);
  \coordinate (B) at (0.3,1.2);
  \coordinate (C) at (1.2,-0.8);
  \coordinate (D) at (-1.5,-0.7);
  \draw [->] (A) -- (B);
  \draw [->] (A) -- (C);
  \draw [->] (A) -- (D);
  \draw (300:2cm) -- (0,0) -- (240:2cm) arc[start angle=240, end angle=300, radius=2cm] -- cycle;
  \draw (180:2cm) -- (0,0) -- (120:2cm) arc[start angle=120, end angle=180, radius=2cm] -- cycle;
  \draw (60:2cm) -- (0,0) -- (0:2cm) arc[start angle=0, end angle=60, radius=2cm] -- cycle;
  
  \foreach \Point/\PointLabel in {(0.3,1.2)/X_1, (1.2,-0.8)/X_2, (-1.5,-0.7)/X_3}
        \draw[fill=red] \Point circle (0.02) node[above ] {$\PointLabel$};
  \draw[fill=red] (0,0) circle (0.02) node [below = 1mm of A] {$X$};
\end{tikzpicture}
\caption{Construction in $d=2$.}
\label{fig:1}
\end{figure}
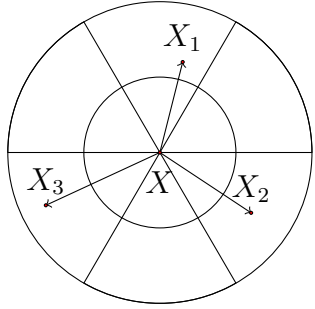

\medskip

{\sc Step 5}. All Poisson points are transported over distances $\ll R$.
We claim that for all Poisson points $X$
\begin{align}\label{fw21}
|T(X)-X|\lesssim \eps R\quad\mbox{provided}\;X\in(-R,R)^d.
\end{align}
Given the Poisson point $X\in(-R,R)^d$, let $\{X_n\}_{n=1}^{d+1}$ be as in Step 4.
By cyclical monotonicity \eqref{fw01}, which implies monotonicity of the map $T$, we have $(T(X_n)-T(X))\cdot(X_n-X)\ge 0$, which we use in form of
\begin{align}\label{fw22}
(T(X)-X)\cdot(X_n-X)&\le(T(X_n)-X_n)\cdot(X_n-X)+|X_n-X|^2\nonumber\\
&\lesssim|T(X_n)-X_n|^2+|X_n-X|^2.
\end{align}
We now appeal to (\ref{fw19}), 
which by (\ref{fw14}) and and (\ref{fw12}) implies
\begin{align*}
|T(X_n)-X_n|\le r.
\end{align*}
Inserting this and (\ref{fw18}) into (\ref{fw22}), we obtain
\begin{align*}
(T(X)-X)\cdot\frac{X_n-X}{|X_n-X|}\lesssim r
\end{align*}
for all $n=1,\cdots,d+1$. Since by (\ref{fw20}), any unit vector $e$ can be written as 
a linear combination of $\{\frac{X_n-X}{|X_n-X|}\}_{n=1}^{d+1}$ with non-negative
weights $\le\frac{1}{\rho}\sim 1$, this implies $|T(X)-X|\lesssim r$. Since (\ref{fw12}) was the
only constraint on $r$, we obtain (\ref{fw21}).
\end{proof}

% Since $B_R\subset Q_R$ we can integrate the $L^\infty$ esimtate to obtain 
% $$ \frac{1}{R^d}\int_{B_R\times \R^d} |x-y|^2 q(dx,dy) \leq o(R^2)$$
% which will be a key ingredient in the next step.

\subsection{Key step: Harmonic approximation}
\begin{lemma}\label{lem:harmonicapprox}
There exist a constant $C$\footnote{here and throughout the paper $C$ denotes a positive constant, which may only depend on the dimension, the value of which may change from line to line.} and a random radius $r_* < \infty$ a.~s. such that for every $R\geq r_*$ we have
\begin{equation}\label{eq:L2statement}
\frac{1}{R^d}\sum_{X \in B_R \; \mbox{or} \; T \bra{X} \in B_R} |T\bra{X}  - X|^2  \leq  C \ln R.
\end{equation}
\end{lemma}
\begin{proof}[Proof of Lemma \ref{lem:harmonicapprox}]
The proof relies on the harmonic approximation result from \cite[Theorem 1.4]{GHO}. This result establishes that for any $0<\tau\ll 1$, there exists
an $\eps>0$ and a $C_\tau<\infty$ such that provided for some $R$
\begin{equation}\label{eq:hypharmapprox}
\frac1{R^2} E(6R) + \frac1{R^2} D(6R)\leq \eps
\end{equation}
(recall \eqref{eq:energy} and \eqref{eq:data} for the definition) there exists a harmonic gradient field $\Phi$ such that
\begin{equation}\label{eq:harmapproxoutput}
\begin{split}
&\frac{1}{R^d} \sum_{X\in B_R\;\mbox{or}\;T(X)\in B_R}  \left|T \bra{X} - X -\nabla \Phi(X)\right|^2 \leq \tau E(6R) + C_\tau D(6R), \\
&  \sup_{B_{2R}} |\nabla \Phi|^2 \leq C_\tau \bra{E(6R) + D(6R)}.
\end{split}
\end{equation}
The fraction $\tau$ will be chosen at the end of the proof. Note that in \eqref{eq:data} $D \bra{R}$ is defined on boxes $(-R,R)^d$ while \cite[Theorem 1.4]{GHO} \eqref{eq:hypharmapprox} requires balls. Since $B_{6R} \subseteq (-6R, 6R)^d$ we may assume that \eqref{eq:hypharmapprox} holds also for $B_{6R'}$ for $R'$ close to $R$, see \cite[Lemma 2.10]{GHO} for this type of restriction property.

%%%%%%%%%%%%
%Recall $E(R)$ and $D(R)$ from \eqref{eq:energy} and \eqref{eq:data}. We want to apply the harmonic approximation theorem from \cite[Theorem 1.4]{GHO}. To this end, we need to ensure the main condition that for $\eps >0$ we have
%\begin{equation}\label{eq:hypharmapprox}
%\frac1{R^2} E(6R) + \frac1{R^2} D(6R)\leq \eps.
%\end{equation}
%By Poisson concentration (see Lemma \ref{lem:Poiconcentration} ) we can assume that for $R\geq r_*$ we have $\frac{\mu(B_R)}{R^d}\in \sqa{\frac12,2}.$ Fix $\eps>0$. 
\medskip
{\sc Step 1}. Definition of $r_*$ depending on $\tau$. For $0< \tau \ll 1$ let $\eps = \eps \bra{\tau}$ be as above. By Lemma \ref{lem:upperboundpathwise} we may assume that $r_*$ is large enough so for any dyadic $R\geq r_*$
\begin{equation}\label{eq:L2upperdata}
D(R) \leq C \ln R
\end{equation}  
and
\begin{equation}\label{eq:L2bounddata}
D(6R) \leq \frac\eps2 R^2.
\end{equation}
Note that only the bound \eqref{eq:L2upperdata} is specific to $d=2$. From now on, we restrict ourselves to dyadic $R$, which we may do w.~l.~o.~g.~ for \eqref{eq:L2statement}. Note that by the bound on $D\bra{6R}$ in \eqref{eq:L2bounddata} and the second and fourth term in the definition of $D\bra{R}$ in \eqref{eq:data}
\begin{equation}\label{eq:L2dataestpoint}
\# \bra{\cur{X \in B_R} \cup \cur{ T(X) \in B_R}} \leq C R^d.
\end{equation}
Moreover, we may assume that $r_*$ is large enough so that \eqref{eq:Linfty} holds. Since $B_{6R}\subset (-6R,6R)^d$ we may sum \eqref{eq:Linfty} over $B_R$ to obtain for $R\geq r_*$ 
\[
\frac{1}{R^d}\sum_{X\in B_{6R}} |T \bra{X} - X|^2 \leq \frac\eps4 R^2.
\]
By symmetry, potentially enlarging $r_*$, we may also assume that \eqref{eq:Linfty} holds with $X$ replaced by $T \bra{X}$ so that both
\begin{equation}\label{eq:L2preimageingoodball}
T \bra{X} \in B_R \Rightarrow X \in B_{2R}
\end{equation}
and
\[
\frac{1}{R^d}\sum_{T\bra{X}\in B_{6R}} |T \bra{X} - X|^2 \leq \frac\eps4 R^2,
\]
thus
$$ E \bra{6R} = \frac{1}{R^d}\sum_{X\in B_{6R}\;\mbox{or}\;T(X)\in B_{6R}} |T \bra{X} - X|^2 \leq \frac\eps2 R^2,$$
and in particular \eqref{eq:hypharmapprox} holds.
Finally by Lemma \ref{lem:L0} we may assume, possibly enlarging $r_*$, that there exists a deterministic constant $L_\tau$ and for $R\geq r_*$ we both have 
\begin{equation}\label{eq:L2ergodic}
\# \bra{\cur{X \in Q_R \ | \ \abs{T \bra{X} - X} > L_\tau} \cup \cur{T(X) \in Q_R \ | \ \abs{T \bra{X} - X} > L_\tau}} \leq \tau R^d.
\end{equation}
and
\begin{equation}\label{eq:L2LleqlnR}
L_\tau^2 \leq \ln R.
\end{equation}
%Finally, by symmetry we may assume, eventually enlarging $r_*$, that \eqref{eq:Linfty} holds with $T \bra{X}$ replaced by $X$. This in particular implies that
%\begin{equation}\label{eq:L2preimageingoodball}
%T \bra{X} \in B_R \Rightarrow X \in B_{\bra{1+\eps}R}.
%\end{equation}

\medskip
{\sc Step 2.} 
Application of harmonic approximation. For all $R \geq r_*$
\begin{equation}\label{eq:L2basestep}
E \bra{R} \leq {\tau} E \bra{6 R} + C_\tau \ln R.
\end{equation}
We split the sum according to whether the transportation distance is moderate or large. On the latter we use the harmonic approximation:
%Since \eqref{eq:hypharmapprox} holds, we may apply the harmonic approximation to get a constant $C_\tau$ and a harmonic gradient field $\Phi$ such that
%$$ \frac{1}{|B_R|} \sum_{X\in B_R\;\mbox{or}\;T(X)\in B_R}  \left|T \bra{X} - X -\nabla \Phi(X)\right|^2 \leq \tau E(6R) + C_\tau D(6R)$$
%and
%$$R^2 \sup_{B_{2R}} |\nabla^2\Phi|^2 + \sup_{B_{2R}} |\nabla \Phi|^2 \leq C_\tau \bra{E(6R) + D(6R)}.$$
%Consider the set of Poisson points  which are transported by a distance distance greater than $L$.
%Hence, we can estimate
\begin{align*}
 \lefteqn{\frac1{R^{d}} \sum_{\bra{X\in B_R\;\mbox{or}\;T(X)\in B_R}\:\mbox{and}\; \abs{T\bra{X}-X}>L_\tau} |T\bra{X} - X|^2}\\
 & \leq  \frac2{R^{d}} \sum_{{X\in B_R\;\mbox{or}\;T(X)\in B_R}} |T \bra{X} -X -\nabla\Phi \bra{X}|^2  \\
 & + \frac2{R^{d}} \sum_{\bra{X\in B_R\;\mbox{or}\;T(X)\in B_R}\:\mbox{and}\; \abs{T\bra{X}-X}>L_\tau} |\nabla\Phi(X)|^2 \\
%& \stackrel{\eqref{eq:L2ergodic}}{\leq} 2\bra{\tau E(6R) + C_\tau D(6R)} + 2\tau \sup_{B_R} |\nabla \Phi|^2\\
& \stackrel{\eqref{eq:harmapproxoutput}, \eqref{eq:L2preimageingoodball}, \eqref{eq:L2ergodic}}{\leq}  2 \tau E(6R) + 2 C_\tau D(6R) + 2\tau C_\tau (E(6R)+D(6R))\\
 & =  2 \tau \bra{1 + C_\tau} E(6R) + 2 C_\tau \bra{1+\tau} D(6R).
\end{align*}
%Note that by the bound on $D\bra{6R}$ in \eqref{eq:hypharmapprox} and the second and fourth term in the definition of $D\bra{R}$ in \eqref{eq:data}
%\begin{equation*}
%\# \bra{\cur{X \in B_R} \cup \cur{ T(X) \in B_R}} \leq C R^d.
%\end{equation*}
%Indeed, by definition \eqref{eq:data} and \eqref{eq:hypharmapprox} we may deduce
%\begin{align*}
%\# \cur{X \in B_R} \leq \bra{1 + \sqrt{\eps}} \abs{B_R} < 2 \abs{B_R} && \# \cur{T \bra{X} \in B_R} \leq \bra{1 + \sqrt{\eps}}\abs{B_R} < 2 \abs{B_R}.
%\end{align*}
The last estimate combines to
\begin{align*} 
\lefteqn{\frac1{R^d}\sum_{X\in B_R\;\mbox{or}\;T(X)\in B_R}|T \bra{X} - X|^2}\\  & =  \frac1{R^d} \sum_{\bra{X\in B_R\;\mbox{or}\;T(X)\in B_R}\:\mbox{and}\; \abs{T\bra{X}-X}\leq L_\tau} |T \bra{X} - X|^2 \\
& +  \frac1{R^d}\sum_{\bra{X\in B_R\;\mbox{or}\;T(X)\in B_R}\:\mbox{and}\; \abs{T\bra{X}-X}>L_\tau} |T \bra{X}-X|^2\\
& \stackrel{\eqref{eq:L2dataestpoint}}{\leq} C L_\tau^2 +2 \tau \bra{1 + C_\tau} E(6R) + 2 C_\tau \bra{1+\tau} D(6R) \\
& \stackrel{\eqref{eq:L2upperdata},\eqref{eq:L2LleqlnR}}{\leq} 2 \tau \bra{1 + C_\tau} E \bra{6R} + \bra{2 C_\tau \bra{1 + \tau}+C} \ln R.
\end{align*}
Relabeling $\tau$ and $C_\tau$, this implies \eqref{eq:L2basestep}.

\medskip
{\sc Step 3.}
Iteration. Iterating \eqref{eq:L2basestep}, we obtain for any $k\geq 1$
\begin{align*}
 E(R) &\leq \tau E(6R) + C_\tau \ln R \\
 & \leq  \tau^2 E(6^2R) + \tau C_\tau \ln R + C_\tau \ln R\\
 &\leq \tau^k E(6^k R) + C_\tau\sum_{l=0}^{k-1}  \tau^l \ln R \\
 & \stackrel{\eqref{eq:hypharmapprox}}{\leq} \eps \bra{36 \tau}^k R^2 + C_\tau\sum_{l=0}^{k-1}  \tau^l \ln R.
\end{align*}
%Since by \eqref{eq:hypharmapprox} $E(R)\leq \eps R^2$,  by \eqref{eq:L2upperdata} choosing $\tau$ such that $36 \tau < 1$, we have
%$$ \tau^k E(6^kR) \leq \tau^k \eps 6^{2k} R^2 \to_{k\to\infty} 0.$$
We now fix $\tau$ such that $36\tau < 1$ to the effect of
\[
E \bra{R} \leq C \sum_{l=0}^{\infty}  \tau^l \ln R \leq C \ln R.
\]
%By \eqref{eq:L2upperdata} applied to $6^l R \geq r_*$ instead of $R$ and noting that in particular $\tau < 1$,  we conclude
%$$\sum_{l=0}^{\infty}  \tau^l D(6^l R) \leq C \ln R.$$
\end{proof}

\subsection{Trading integrability against asymptotics}
\begin{lemma}\label{lem:upperbound}
For every $\eps >0$ there exists a random radius $r_* < \infty$ a.~s. such that
$$ \frac{1}{R^d} \sum_{X\in B_R \;\mbox{or}\;T\bra{X} \in B_R}|T \bra{X} - X|  \leq  \eps  \ln^{\frac12} R.$$
\end{lemma}
\begin{proof}
% By Step 3 we know that for $R\gg 1$
% $$\frac{1}{R^d}\int_{B_R\times\R^d}|x-y|^2 q(dxdy) \leq O(\ln R).$$
% Combining this with the $L^\infty$-estimate Lemma \ref{lem:Linfty} we can estimate for $R\geq r_*$ using the concentration bound Lemma \ref{lem:Poiconcentration} and the Cauchy-Schwarz inequality
By Lemma \ref{lem:harmonicapprox}, we know that there exists a random radius $r_*$ such that for $R\geq r_*$ we have 
\begin{equation}\label{eq:L1L2est}
E \bra{R} = \frac{1}{R^d}\sum_{X\in B_R \;\mbox{or}\; T\bra{X} \in B_R}|T\bra{X} - X|^2  \leq C\ln R.
\end{equation}
Let $0 < \eps \ll 1$. Possibly enlarging $r_*$, we may also assume by Lemma \ref{lem:L0} that there exists a deterministic constant $L$ such that for $R \geq r_*$
\begin{equation}\label{eq:L1largetransport}
\# \bra{\cur{X \in B_R \ | \ \abs{T \bra{X} - X} > L} \cup \cur{T(X) \in B_R \ | \ \abs{T \bra{X} - X} > L}} \leq \eps R^d.
\end{equation}
Furthermore, note that by Lemma \ref{lem:upperboundpathwise} and the second and fourth term in the definition of $D \bra{R}$ in \eqref{eq:data} we may also assume possibly enlarging $r_*$ again that for $R \geq r_*$ \eqref{eq:L2dataestpoint} holds.
%\begin{equation}\label{eq:L1pointsball}
%\# \bra{\cur{X \in B_R} \cup \cur{ T(X) \in B_R}} \leq C R^d.
%\end{equation}
Finally, we may also assume possibly enlarging $r_*$ that for $R \geq r_*$
\begin{equation}\label{eq:L1Lsmall}
L \leq \eps^\frac12 \ln^\frac12 R.
\end{equation}
We split again the sum into moderate and large transportation distance and apply Cauchy-Schwarz:
\begin{align*}
\frac{1}{R^d}\sum_{X\in B_R \;\mbox{or} \; T\bra{X} \in B_R} |T\bra{X}-X| & \leq  \frac{1}{R^d}\sum_{\bra{X\in B_R\;\mbox{or}\;T(X)\in B_R}\:\mbox{and}\; \abs{T\bra{X}-X}\leq L}  |T\bra{X} - X| \\
& +\frac{1}{R^d}\sum_{\bra{X\in B_R\;\mbox{or}\;T(X)\in B_R}\:\mbox{and}\; \abs{T\bra{X}-X}>L}  |T\bra{X} - X|\\
& \stackrel{\eqref{eq:L2dataestpoint}, \eqref{eq:L1L2est}, \eqref{eq:L1largetransport}}{\leq} C L + \eps^\frac12 E(R)^\frac12 \\
&\stackrel{\eqref{eq:L1Lsmall}}{\leq} C \eps^\frac12 \ln^\frac12 R.
\end{align*}
Relabeling $\eps$ proves the claim.
\end{proof}

\subsection{Asymptotics for the bipartite matching problem in dimension $d=2$}\label{sec:bounds}
In this section we give a self-contained proof of the upper and lower bounds on the asymptotics in the matching problem in the critical dimension $d=2$. 

\medskip
The intuition why $d=2$ is critical for optimal transportation is clear:
the fluctuations of the number density of the Poisson point process of unit intensity,
i.~e.~its deviation from $1$ on some mesoscopic scale $r\gg 1$
are of size $O(\frac{1}{\sqrt{r^d}})$. To compensate these fluctuations, one has to displace
particles by a distance $O(r\times \frac{1}{\sqrt{r^d}})$, which is $O(1)$ iff $d=2$.
Hence taking care of the fluctuations on every dyadic scale $r$ requires a
displacement of $O(1)$. Naively, this suggests a transportation
cost per particle that is logarithmic in the ratio between the macroscopic scale $R$
and the microscopic scale $1$. However, by the independence properties of the Poisson point process,
the displacements on every dyadic scale are essentially independent, 
so that there are the usual cancellations when adding up the dyadic scales. 
Hence the displacement behaves like the {\it square root} of the logarithm.

\medskip
The proof of the lower bound is similar to the original proof in \cite{AKT84}. However, in the final step we use martingale arguments instead of Gaussian embeddings. The proof of the upper bound has some similarities to the proof in \cite{AmStTr16}. Again we use martingale arguments and crucially the Burkholder-Davis-Gundy inequality. It would be interesting to see whether our technique for the upper bound allows to get the precise constant as in \cite{AmStTr16}.

\begin{lemma}\label{lem:lowerboundpathwise}
Let $\mu, \nu$ denote two independent Poisson point processes in $\mathbb{R}^2$ of unit intensity. There exists a constant $C$, and a random radius $r_* < \infty$ a.~s. such that for any dyadic radii $R \geq r_*$
\[
\sup\cur{\int\zeta \bra{d\mu-d\nu}
\mbox{}\quad\big|\;{\rm supp}\zeta\subset(0,R)^2,\;|\nabla\zeta|\le 1,\;
\int\zeta dx=0} \geq C R^2 \ln^\frac12 R.
\]
\end{lemma}

\begin{lemma}\label{lem:upperboundpathwise}
%Let $\mu$ denote the Poisson point process in $\mathbb{R}^2$ of unit density.
Let $\mu, \nu$ be as in Lemma \ref{lem:lowerboundpathwise}. Then there exists a constant $C$, and there exists a random radius $r_*<\infty$ a.~s. such that for any dyadic radii $R \geq r_*$
\[
D(R) \leq C \ln^\frac12 R.
\]
\end{lemma}

\medskip
For the lower as well as the upper bound, our proof proceeds in two stages. First we show the desired estimate in expectation. Then we use concentration arguments to lift it to a pathwise estimate.

\subsubsection{Lower bound on $W_1$}
% 
% We first show that \eqref{eq:lowerbd} is true in expectation and then we lift this result to a pathwise estimate by concentration properties of the Poisson point process. We start by

\begin{lemma}\label{lem:dualLP}
 Let $\mu$ denote the Poisson point process in $\mathbb{R}^2$ of unit intensity.
Then it holds for $R\gg 1$
\begin{align}\label{eq:dualPL}
\mathbb{E}{\sup\cur{\int\zeta d\mu
\mbox{}\quad\big|\;{\rm supp}\zeta\subset(0,R)^2,\;|\nabla\zeta|\le 1,\;
\int\zeta dx=0}}\gtrsim R^2\ln^\frac{1}{2}R.
\end{align}
\end{lemma}

%Before giving the proof of Lemma \ref{lem:dualLP} we show how it implies the pathwise result of Lemma \ref{lem:lowerboundpathwise} by concentration properties \cm{? inequalities for ?} of the Poisson process.

%\newpage

\begin{proof}
 W.~l.~o.~g.~we may assume that $R\in 2^\mathbb{N}$, and will consider the (finite) family of all
dyadic cubes $Q\subset (0,R)^2$ of side-length $\ge 1$. 
For such $Q$, we call $Q_r$ and $Q_l$ the right and left half of $Q$ and consider the integer-valued random variable
\begin{align}\label{ao14}
N_Q=\mu(Q_r)-\mu(Q_l).
%N_Q:=\mbox{number in the right half of $Q$}-\mbox{number in the left half of $Q$}.
\end{align}
Fixing a smooth mask (or reference function) $\hat\zeta$ with
\begin{align}\label{ao06}
{\rm supp}\hat\zeta\in(0,1)^2,\quad\int\hat\zeta\,d\hat x=0,\quad
\int_{\hat x_1>\frac{1}{2}}\hat\zeta d\hat x
-\int_{\hat x_1<\frac{1}{2}}\hat\zeta d\hat x=1,
\end{align}
for every $Q$ we define $\zeta_Q$ via the simple transformation
\begin{align}\label{ao04}
\zeta_Q(A_Q\hat x)=\hat\zeta(\hat x)\;\mbox{where}\;A_Q\;\mbox{is the affine map with}\;Q=A_Q(0,1)^2.
\end{align}
For later purpose we note that by standard properties of the Poisson process (see \cite[equation (4.23)]{lastpenrose}) we have
\begin{align}\label{ao02}
\mathbb{E}N_Q^2=|Q|,\quad\mathbb{E}N_QN_{Q'}=0\;\mbox{for}\;Q\not=Q',\quad
\mathbb{E}N_Q\int\zeta_Qd\mu=|Q|.
\end{align}
%
% Since $N_Q=\int\chi_Qd\mu$, where $\chi_Q=\pm 1$ on the right/left half of $Q$
% (and zero else), {\color{blue} too complicated; follows directly from Poisson calculation, as in \eqref{eq:covPoi}} this follows from the fact that second moments of $\int\zeta d\mu$ 
% for a function function $\zeta$ with $\int\zeta dx=0$ agree with second moments
% of $\int\zeta\xi$ for white noise $\xi$. 
% Hence the first expectation is indeed given by $\int\chi_Q^2=|Q|$,
% the second one by $\int\chi_Q\chi_{Q'}=0$, and the third by $\int\chi_Q\zeta_Q=|Q|$;
% the latter follows from the last identity in (\ref{ao06}).
% 
% \medskip

For every dyadic $r=R,\frac{R}{2},\frac{R}{4},\cdots,1$ we consider the function
\begin{align}\label{ao05}
\zeta_r:=\sum_{Q\;\mbox{of level}\;\ge r}N_Q\zeta_Q
\end{align}
and note that for a fixed point $x\in(0,R)^d$, we have
\begin{align}\label{ao25}
\nabla\zeta_r(x)=\sum_{R\ge \rho\ge r}\sum_{Q\;\mbox{of level}\;\rho}N_Q\nabla\zeta_Q(x),
\end{align}
observing that the sum over $Q$ restricts to the one cube
of dyadic level/side-length $\rho$ that contains $x$. 
We now argue that (\ref{ao25}) is a martingale, where $r=R,\frac{R}{2},\frac{R}{4},\cdots$,
plays the role of a discrete time. More precisely, it is a martingale
with respect to the filtration generated by $\{\{N_Q\}_{Q\;\mbox{of level}\;r}\}_r$.
%Indeed, in view
%
%\begin{align*}
%\nabla\zeta_r(x)\stackrel{(\ref{ao05})}{=}
%\sum_{r\le\rho\le R}\sum_{Q\;\mbox{of level}\;\rho}N_Q\nabla\zeta_Q(x),
%\end{align*}
%
It suffices to show that
even when conditioned on $\{N_{\bar Q}\}_{\bar Q\;\mbox{of level}\;\rho}$
for $\rho\ge 2r$, the expectation of $N_Q$ for every $Q$ of level $r$ vanishes.
This can be seen by appealing to a Lebesgue-measure preserving
transformation of $\mathbb{R}^2$ that swaps the left and right half of $Q$;
when applied to point configurations, it preserves the law of the Poisson point process,
it leaves $N_{\bar Q}$ for $\bar Q$ of level $\rho\ge 2r$ invariant, and it
converts $N_Q$ into $-N_Q$.

\medskip

As the sum over $Q$ in (\ref{ao25}) reduces to a single summand, this
martingale has (total) quadratic variation\footnote{here we consider the total quadratic variation and not the more involved quadratic variation since it is sufficient for our purposes.} 
\begin{align*}
\sum_{R\ge r\ge 1} \sum_{Q\;\mbox{of level}\;r}  N_Q^2|\nabla\zeta_Q(x)|^2,
\end{align*}
and we claim that it satisfies
\begin{align}\label{ao03}
\mathbb{E}\int\sum_{R\ge r\ge 1}\sum_{Q\;\mbox{of level}\;r}N_Q^2|\nabla\zeta_Q(x)|^2dx
\lesssim R^2\ln R.
\end{align}
Indeed, by (\ref{ao02}) we obtain that the l.~h.~s.~of (\ref{ao03})
is given by $\sum_{Q}$ $\int|\nabla\zeta_Q(x)|^2dx$.
By (\ref{ao04}), this in turn is estimated by $\sum_{Q}$ $|Q|$
$=\sum_{R\ge r\ge 1}R^2$ $=R^2\log_2R$.
We also claim that the last item in (\ref{ao02}) implies
\begin{align}\label{ao13}
\mathbb{E}\int\zeta_1d\mu\gtrsim R^2\ln R.
\end{align}
Indeed, the l.~h.~s.~of (\ref{ao13}) is again given by $\sum_{Q}$ $|Q|$.
We now appeal to the (quadratic) Burkholder inequality \cite[Theorem 6.3.6]{Stroock} (exchanging $\mathbb{E}$ and
$\int dx$) to obtain from (\ref{ao03})
\begin{align}\label{ao11}
\mathbb{E}\int\sup_{R\ge r\ge 1}|\nabla\zeta_r(x)|^2dx\lesssim R^2\ln R.
\end{align}
%
%This follows from exchanging $\mathbb{E}$ and $\int dx$ and arguing that for
%a fixed point $x$, 
%$\{\nabla\zeta_r(x)\}_r$ is a Martingale, where $r=R,\frac{R}{2},\frac{R}{4},\cdots$,
%plays the role of a discrete time. More precisely, it is a Martingale
%with respect to the filtration generated by $\{\{N_Q\}_{Q\;\mbox{of level}\;r}\}_r$. 
%Indeed, in view of (\ref{ao25}) 
%
%\begin{align*}
%\nabla\zeta_r(x)\stackrel{(\ref{ao05})}{=}
%\sum_{r\le\rho\le R}\sum_{Q\;\mbox{of level}\;\rho}N_Q\nabla\zeta_Q(x),
%\end{align*}
%
%it suffices to show that 
%even when conditioned on $\{N_{\bar Q}\}_{\bar Q\;\mbox{of level}\;\rho}$
%for $\rho\ge 2r$, the expectation of $N_Q$ for every $Q$ of level $r$ vanishes.
%This can be seen by appealing to a Lebesgue-measure preserving 
%transformation of $\mathbb{R}^2$ that swaps the left and right half of $Q$;
%when applied to point configurations, it preserves the law of the Poisson point process,
%it leaves $N_{\bar Q}$ for $\bar Q$ of level $\rho\ge 2r$ invariant, and it
%converts $N_Q$ into $-N_Q$.

\medskip

We now are in the position to define $\zeta$ via a stopping ``time''.
Given an $M<\infty$, which we think of as being large and that will be chosen later,
we keep subdividing the dyadic cubes (recall that we restrict to those of side length $\ge 1$)
as long as
\begin{align}\label{ao07}
\int_Q|\nabla\sum_{\bar Q\supset Q}N_{\bar Q}\zeta_{\bar Q}|^2\le M |Q|\ln R\quad
\mbox{and}\quad|N_Q|^2\le M |Q|  \ln R  .
\end{align}
This defines a nested sub-family of dyadic cubes, and thereby a random 
(spatially piecewise constant) stopping scale $r_*=r_*(x)$ (note that $\frac{r_*}{2}$ is a stopping time but we will not need that in this proof). We then set
\begin{align}\label{ao09}
\zeta(x):=\zeta_{r_*(x)}(x).
\end{align}
We first argue that we have the Lipschitz condition
\begin{align}\label{ao08}
|\nabla\zeta|^2\lesssim M\ln R.
\end{align}
Indeed, consider one of the finest cubes $Q$ in the family constructed in (\ref{ao07});
by definition (\ref{ao09}), the first item in (\ref{ao07}) amounts to
\begin{align*}
\int_Q|\nabla\zeta|^2\le M|Q|\ln R,
\end{align*}
so that (\ref{ao08}) follows once we show for arbitrary point $x$
\begin{align}\label{ao10}
r_*(x)|\nabla^2\zeta(x)|\lesssim\sqrt{M \ln R}.
\end{align}
Indeed, by definitions (\ref{ao04}), (\ref{ao05}) and (\ref{ao09}),
the l.~h.~s. $|\nabla^2\zeta(x)|$ of (\ref{ao10}) is estimated by $\sum_{r_*(x)\le r\le R}$
$r^{-2}$ $\sum_{Q\ni x\;\mbox{level}\;r} \abs{N_Q}$. By construction of $r_*(x)$
in form of the second item in (\ref{ao07}) this is estimated by
$\sum_{r_*(x)\le r\le R}$ $r^{-1} \sqrt{M  \ln R}$. This is a geometric series that
is estimated by $\lesssim r_*^{-1}(x) \sqrt{M   \ln R}$, as desired.

\medskip

We now argue that the ``exceptional'' set
\begin{align*}
E:=\{\,x\in(0,R)^2\,|\,r_*(x)>1\,\},
\end{align*}
where the dyadic decomposition stops before reaching the minimal scale $r=1$,
has small volume fraction in expectation:
\begin{align}\label{ao12}
M\mathbb{E}|E|\lesssim R^2.
\end{align}
%
%For later use, we note that by definitions (\ref{ao05}) and (\ref{ao09}) we have
%
%\begin{align*}
%\zeta=\zeta_1\quad\mbox{outside of}\;E.
%\end{align*}
%
Indeed, by definitions (\ref{ao05}) and (\ref{ao07}), 
$E$ is the disjoint union of dyadic cubes $Q$ that have at least one (out of four) children 
$Q'$ of level $r=\frac{r_*(x)}{2}$ with
\begin{align*}
\int_{Q'}|\nabla\zeta_r|^2>M|Q'|\ln R\quad\mbox{or}\quad|N_{Q'}|^2>M|Q'| \ln R,
\end{align*}
which combines to 
\begin{align*}
|N_{Q'}|^2+\int_{Q'}\sup_{1\le r\le R}|\nabla\zeta_r|^2>M\frac{|Q|}{4}\ln R.
\end{align*}
Summing over all $Q$ covering $E$, this yields
\begin{align*}
\sum_{\mbox{all}\;Q}|N_{Q}|^2+\int\sup_{1\le r\le R}|\nabla\zeta_r|^2>M\frac{|E|}{4}\ln R.
\end{align*}
Taking the expectation we obtain from (\ref{ao02}) and (\ref{ao11}) 
\begin{align*}
R^2\ln R\gtrsim M\mathbb{E}|E|\ln R,
\end{align*}
which yields (\ref{ao12}). 

\medskip

We finally argue that
\begin{align}\label{ao15}
M^\frac{1}{4}\mathbb{E}|\int(\zeta_1-\zeta)d\mu|\lesssim R^2\ln R.
\end{align}
Together with (\ref{ao13}) this implies 
$\mathbb{E}\int\zeta d\mu\gtrsim R^2\ln R$ for $M\gg 1$, which we fix now. 
In combination with (\ref{ao08}) this implies the claim of the lemma.
The dyadic decomposition defines a family of exceptional cubes $Q$ (a cube
$Q$ is exceptional iff for any $x\in Q$ we have that its level $r$ satisfies $r<r_*(x)$).
In view of (\ref{ao05}) and (\ref{ao09}), we thus have to estimate
$\sum_{Q} I(Q\;\mbox{exceptional}) N_Q\int \zeta_Q d\mu$.
Applying $\mathbb{E}|\cdot|$ and using H\"older's inequality in probability
we obtain
\begin{align*}
\mathbb{E}|\int(\zeta_1-\zeta)d\mu|\le
\sum_Q (\mathbb{P}Q\;\mbox{exceptional})^\frac{1}{4}
(\mathbb{E}N_Q^4)^\frac{1}{4} (\mathbb{E}(\int\zeta_Q d\mu)^2)^\frac{1}{2}.
\end{align*}
As for (\ref{ao02}), we have for the last factor 
$\mathbb{E}(\int\zeta_Q d\mu)^2$ $=\int\zeta_Q^2$ $\lesssim |Q|$.
For the middle factor, we recall that the two numbers in (\ref{ao14}) are
Poisson distributed with mean $\frac{1}{2}|Q|$, so that
$\mathbb{E}N_Q^4$ $\lesssim|Q|^2$ by elementary properties of the Poisson distribution. 
Hence we gather
\begin{align*}
\mathbb{E}|\int(\zeta_1-\zeta)d\mu|
&\lesssim\sum_Q (\mathbb{P}Q\;\mbox{exceptional})^\frac{1}{4}|Q|\nonumber\\
&\le\big(\sum_Q (\mathbb{P}Q\;\mbox{exceptional})|Q|\big)^\frac{1}{4}
\big(\sum_Q|Q|\big)^\frac{3}{4}.
\end{align*}
As we noted before, we obtain for the second factor $\sum_Q|Q|\lesssim R^2\ln R$.
For the first factor, we note 
\begin{align}\label{ao16}
\sum_{Q\;\mbox{level}\;r}(\mathbb{P}Q\;\mbox{exceptional})|Q|
=\mathbb{E}|\cup_{Q\;\mbox{exceptional level}\;r}Q|. 
\end{align}
Since $\cup_{Q\;\mbox{exceptional}}Q\subset E$,
we have that the sum in $r$ of the l.~h.~s. of (\ref{ao16}) is 
$\lesssim\mathbb{E}|E|\ln R$. Now (\ref{ao15}) follows from (\ref{ao12}).
\end{proof}

\begin{proof}[Proof of Lemma \ref{lem:lowerboundpathwise}]
\medskip
{\sc Step 1}. From one to two measures. We claim that for $R\gg 1$ the following inequality holds
\begin{equation}\label{eq:dualPP}
\mathbb{E}{\sup\cur{\int\zeta (d\mu - d\nu)
\mbox{}\quad\big|\;{\rm supp}\zeta\subset(0,R)^2,\;|\nabla\zeta|\le 1,\;
\int\zeta dx=0}}\gtrsim R^2\ln^\frac{1}{2}R.
\end{equation}
Indeed, write $\mathcal F=\cur{\zeta : {\rm supp}\zeta\subset(0,R)^2,\;|\nabla\zeta|\le 1,\;
\int\zeta dx=0 }$ and $S=\sup_{\zeta\in\mathcal F} \int \zeta d\mu,$ where the supremum has to be understood as an essential supremum. By basic results on essential suprema, there exists a countable subset $\{\zeta_n, n\geq 1\}\subset \mathcal F$ such that $S=\sup_{n\geq 1} \int \zeta_n d\mu$. Setting, $S_n=\max_{k\leq n} \int\zeta_k d\mu$ we have $S_n\nearrow S$ and by monotone convergence $\EE S_n\nearrow \EE S.$ Hence, there is $n$ such that $\EE S_n\ges R^2\ln^\frac12 R$  by Lemma \ref{lem:dualLP}. By definition of $S_n$, $S_n=\int \zeta d\mu$ for some $\mathcal F$-valued random variable $\zeta$ (that takes the finitely many values $\zeta_1, \dots, \zeta_n$), which is admissible in \eqref{eq:dualPP}. Clearly, $\zeta$ is measurable w.~r.~t. $\mu$ alone. However, $\zeta$ is measurably only dependent on $\mu$ and not on $\nu$. Since $\nu$ has intensity measure $dx$ we obtain by independence of $\mu$ and $\nu$
\begin{align*}
 \EE \sqa{\int \zeta (d\mu-d\nu)}& = \EE_\mu\sqa{\EE_\nu\sqa{\int \zeta (d\mu-d\nu)}} = \EE_\mu\sqa{\int \zeta d\mu - \int \zeta dx}\\
 &= \EE_\mu\sqa{\int \zeta d\mu}\ges R^2 \ln^\frac12 R
\end{align*}
which proves \eqref{eq:dualPP}. 

\medskip
{\sc Step 2}. Concentration around expectation. 
We claim that
\begin{equation}\label{eq:GRmunu}
S_R \bra{\mu, \nu} := \sup\cur{\int\zeta \bra{d\mu-d\nu}
\mbox{}\quad\big|\;{\rm supp}\zeta\subset(0,R)^2,\;|\nabla\zeta|\le 1,\;
\int\zeta dx=0}
\end{equation}
satisfies
\begin{align}\label{eq:inconlem}
\lim_{\substack{R \rightarrow \infty \\ \mbox{$R$ dyadic}}} \frac{1}{R^2 \ln^{\frac{1}{2}} R} \abs{S_R \bra{\mu, \nu} - \mathbb{E} S_R \bra{\mu, \nu}} = 0 \quad \P - \text{a.~s.}.
\end{align}
Indeed, by the triangle inequality, we split the statement into two:
\[
\abs{S_R \bra{\mu, \nu} - \mathbb{E} S_R \bra{\mu, \nu}} \leq \abs{S_R \bra{\mu, \nu} - \mathbb{E}_\mu S_R \bra{\mu, \nu}} + 
\abs{\EE_\mu S_R \bra{\mu, \nu} - \mathbb{E} S_R \bra{\mu, \nu}}.
\]
By a Borel-Cantelli argument it suffices to show that for any fixed $\eps>0$ and $R \gg 1$ the following statements hold
\begin{equation}\label{eq:PrInconcentratiomu}
\mathbb{P} \bra{ \frac{1}{R^2 \ln^\frac{1}{2}R} \abs{S_R \bra{\mu, \nu} - \mathbb{E}_\mu S_R \bra{\mu, \nu}} > \eps} \lesssim  \exp\bra{ - \frac{\eps^2}{4} \ln R},
\end{equation}
\begin{equation}\label{eq:PrInconcentrationu}
\mathbb{P} \bra{ \frac{1}{R^2 \ln^\frac{1}{2}R} \abs{\EE_\mu S_R \bra{\mu, \nu} - \mathbb{E} S_R \bra{\mu, \nu}} > \eps} \lesssim  \exp\bra{ - \frac{\eps^2}{4} \ln R}.
\end{equation}
We first turn to \eqref{eq:PrInconcentratiomu}. For $z \in \mathbb{R}^2$ we consider the difference operator 
\[
D_{z} S_R \bra{\mu, \nu} := S_R \bra{\mu + \delta_z, \nu} - S_R \bra{\mu, \nu}.
\]
Since the $\zeta$'s in the definition \eqref{eq:GRmunu} satisfy $\sup \abs{\zeta} \leq R$, we have
\begin{equation}\label{eq:lowerdiffopest}
\abs{D_{z} S_R \bra{\mu, \nu}} \leq \sup\cur{\abs{\zeta \bra{z}}
\mbox{}\big|\;{\rm supp}\zeta\subset(0,R)^2,\;|\nabla\zeta|\le 1,\;
\int\zeta dx=0} \leq R =: \beta.
\end{equation}
Hence,
\[
\int_{\R^2} \bra{D_{z} S_R \bra{\mu, \nu}}^2 dz \leq \int_{\bra{0,R}^2} R^2 dz \leq R^4=: \alpha^2.
\]
Applying \cite[Proposition 3.1]{Wu00} we obtain for $R \gg 1$
\[
\begin{split}
\mathbb{P} \bra{\abs{S_R \bra{\mu, \nu} - \mathbb{E}_\mu S_R \bra{\mu, \nu}} > \eps R^2 \ln^\frac{1}{2}R} & \lesssim \exp \bra{- \frac{\eps R^2 \ln^\frac12 R}{2 \beta} \ln\bra{1+ \frac{\beta \eps R^2 \ln^\frac12 R}{\alpha^2}}}\\
& = \exp \bra{-  \frac{\eps R \ln^{\frac{1}{2}} R}{2} \ln \bra{1 + \frac{\eps \ln^{\frac{1}{2}} R}{R}}} \\
& \lesssim  \exp \bra{- \frac{\eps^2}{4} \ln R}.
\end{split}
\]
The argument for \eqref{eq:PrInconcentrationu} is almost identical: For arbitrary $w \in \R^2$ we need to consider
\[
D_w \EE_\mu S_R \bra{\mu, \nu} = \EE_\mu S_R \bra{\mu, \nu + \delta_w} - \EE_\mu S_R \bra{\mu, \nu}.
\]
Because of $D_w \EE_\mu S_R \bra{\mu, \nu} = \EE_\mu D_w S_R \bra{\mu, \nu}$ we obtain as above
\[
\abs{D_w \EE_\mu S_R \bra{\mu, \nu}} = \abs{\EE_\mu D_w S_R \bra{\mu, \nu}} \leq R. 
\]
Then applying once more \cite[Proposition 3.1]{Wu00} with $\beta = R$ and $\alpha = R^2$ implies \eqref{eq:PrInconcentrationu}.
\end{proof}

\subsubsection{Upper Bound}

\begin{lemma}\label{Lem2}
Let $\mu$ denote the Poisson point process in $\mathbb{R}^2$ of unit intensity.
Then it holds for $R\gg 1$
\begin{equation}\label{eq:upperboundmean}
\mathbb{E}W_{(0,R)^2}^2(\mu,n)\lesssim R^2\ln R,
\end{equation}
where $n=\frac{\mu((0,R)^2)}{R^2}$ is the (random) number density.
\end{lemma}

%As for the lower bound we first show how this result can be lifted to a pathwise estimate by using concentration properties of the Poisson process. Again we will use Wu's result so that we first need to estimate the discrete gradient of $W_{(0,R)^2}(\mu,\kappa)$.

%%%%%%%%%%%%%%%%%%%%%%%%%%%%%%%%%%%%
%%%%%%%%%%%%%%%%%%%%%%%%%%%%%%%%%%%%
%%%%%%%%%%%%%%%%%%%%%%%%%%%%%%%%%%%%
%%%%%%%%%%%%%%%%%%%%%%%%%%%%%%%%%%%%
%%%%%%%%%%%%%%%%%%%%%%%%%%%%%%%%%%%%
%%%%%%%%%%%%%%%%%%%%%%%%%%%%%%%%%%%%
%%%%%%%%%%%%%%%%%%%%%%%%%%%%%%%%%%%%
%%%%%%%%%%%%%%%%%%%%%%%%%%%%%%%%%%%%
%%%%%%%%%%%%%%%%%%%%%%%%%%%%%%%%%%%%
%%%%%%%%%%%%%%%%%%%%%%%%%%%%%%%%%%%%

%%%%%%%%%%%%%%%%%%%%%%%%%%%%%%%%%%%%%%%%%%%%%%%%%%%%%%%%%%%%%%%%%%%%%%%%%%%%%%%%

\begin{proof}
Clearly, the intuition mentioned at the beginning of Section \ref{sec:bounds} suggests to uncover a 
martingale structure. 
W.l.o.g.~we may assume that $R\in 2^\mathbb{N}$, and will consider the family of all
dyadic squares $Q$.
For such $Q$, we consider the number density
\begin{align}\label{ao30}
n_Q:=\frac{\mu \bra{Q}}{|Q|}.
\end{align}
Note that $n_Q|Q|\in\mathbb{N}_0$ is Poisson distributed with expectation $|Q|$.
Note that if, for a fixed point $x$, we consider the sequence of nested squares
$Q$ that contain $x$, the corresponding random sequence $n_Q$ is a martingale. 
We want to stop the dyadic subdivision just before the number density
leaves the interval $[\frac{1}{2},2]$ of moderate values.
This defines a scale: 
\begin{equation}\label{ao32}
r_*(x):= 2 \sup\{r \ | \ \mbox{is the sidelength of dyadic square}\:Q\ni x\;\mbox{with}\;
n_Q\not\in[\frac{1}{2},2]\}.
\end{equation}
%
%Observe that $\frac{r_*(x)}{2}$ is a stopping time w.r.t.\ the filtration generated by $\{n_Q, Q \text{ sidelength }\ge r\}_{R\ge r}$.
Since on a square $Q$ of side length $r=\frac{1}{2}$
we have $n_Q\in 4\mathbb{N}_0$ we trivially obtain
\begin{align}\label{ao42}
r_*(x) \ge 1.
\end{align}
As we shall argue now, it follows from the properties
of the Poisson distribution that (the stationary) $r_*(x)$ is $O(1)$ with overwhelming probability,
in particular
\begin{align}\label{ao33}
\mathbb{E}r_*^4(x)\lesssim 1.
\end{align}
Indeed, by the concentration properties of the Poisson distribution we have for any $\rho \in 2^\N$
\[
\begin{split}
\P \bra{r_*(x) > \rho} & \leq \sum_{Q \ni x, r_Q \geq \rho} \P \bra{n_{Q_\rho} \notin \sqa{\frac12, 2}} \\
& \lesssim \sum_{r \geq \rho\;\mbox{dyadic}} \exp \bra{- C r^2} \lesssim \exp\bra{-C \rho^2}.
\end{split}
\]
Thus we may estimate
\[
\begin{split}
\EE r_*^4 (x) & = 4 \int_0^\infty \rho^{3} \P \bra{r_* (x)> \rho} d\rho  \lesssim \int_0^\infty \rho^{3} \exp \bra{-C \rho^2} d\rho \lesssim 1.
\end{split}
\]

\medskip

We now distinguish the case of $r_* \leq R$ on $(0,R)^2$ and its complement. In the last case, there exists a $y \in (0,R)^2$ such that $r_*(y) > R$, which by \eqref{ao32} means that there exists a dyadic cube $Q \ni y$ of sidelength $r_Q \ge R$ and $n_Q \notin \sqa{\frac12, 2}$, which in turn implies that $r_* > R$ on the entire $(0,R)^2$. Now fix a deterministic $y \in (0,R)^2$. Since $n_{(0,R)^2} R^2$ is the number of particles we use the brutal estimate of the transportation distance 
\begin{equation}\label{eq:uppbrutest}
W^2_{(0,R)^2} \bra{\mu, n_{(0,R)^2}} \leq n_{(0,R)^2} R^22R^2.
\end{equation} 
Since by assumption, $(0,R)^2 \subset (0, r_*(y))^2$, this yields $W^2_{(0,R)^2}$ $\left( \mu, \right.$ $\left. n_{(0,R)^2}\right)$ $\leq n_{(0,r_*(y))^2}$ $r_*^2 (y) 2 R^2$. Hence by definition of $r_*(y)$, see \eqref{ao32} we obtain $W^2_{(0,R)^2} \bra{\mu, n_{(0,R)^2}} \leq 2 r_*^2 (y) 2 R^2 $ $\lesssim r_*^4 (y)$.
By \eqref{ao33}, taking the expectation yields as desired 
\[
\EE W^2_{(0,R)^2} \bra{\mu, n_{(0,R)^2}} I \bra{r_*(y) > R} \lesssim 1,
\]
where $I \bra{r_*(y) > R}$ denotes the indicator function of the event $\cur{r_*(y) > R}$.

\medskip

In the more interesting case of $r_* \leq R$, we define a partition of $(0,R)^2$. A cube $Q_*$ is an element of the partition if and only if its sidelength $r_{Q_*}$ satisfies
\begin{equation}\label{eq:rQstar}
r_{Q_*}=\max\{r_*(x) : x \in Q_*\}.
\end{equation}
In words, this means that the number density of $Q_*$ is still within $[\frac12,2]$ but the number density of at least one of its four children leaves the interval $[\frac12,2]$. This implies the following reverse of \eqref{eq:rQstar}
\begin{equation}\label{eq:reverserstar}
\fint_{Q_*} r^2_{*} \geq \frac14 r^2_{Q_*}.
\end{equation}
%Note that the family $\{Q_*\}$ of those dyadic squares $Q$ just before we stop the 
%subdivision form a partition of $(0,R)^2$.
Equipped with this partition, we define the density $\lambda$ by
\begin{align*}
\lambda=n_{Q_*}\quad\mbox{on}\;Q_*.
\end{align*}
We consider the transportation distance between $\mu$ and the measure $\lambda dx$.
By definition of $\lambda$, there exists a coupling where the mass is only
redistributed within the partition $\{Q_*\}$. This implies the inequality
\begin{align*}
W_{(0,R)^2}^2(\mu,\lambda)\le\sum_{Q_*}2r_{Q_{*}}^2 n_{Q_*}|Q_*|
\stackrel{\eqref{ao32},\eqref{eq:reverserstar}}{\le}16\sum_{Q_*} \int_{Q_*} r^2_{*} = 16\int_{(0,R)^2}r_*^2.
\end{align*}
Taking the expectation, by \eqref{ao33} and Jensen's inequality,
\begin{align*}
\EE W_{(0,R)^2}^2(\mu,\lambda) I \bra{r_*  \leq R}\lesssim 1.
\end{align*}
Hence in view of (\ref{ao33}) and the triangle inequality for the transportation
distance it remains to show
\begin{align}\label{ao38}
\mathbb E W_{(0,R)^2}^2(\lambda,n_{(0,R)^2}) I \bra{r_*  \leq R}\lesssim R^2\ln R.
\end{align}

\medskip

The purpose of the stopping (\ref{ao32}) is that we have the lower bound 
$\lambda, n_{(0,R)^2}\ge\frac{1}{2}$ on the two densities,
which implies the inequality
\begin{align}\label{ao41}
W_{(0,R)^2}^2(\lambda,n_{(0,R)^2})\le 2\int_{(0,R)^2}|j|^2
\end{align}
for all distributional solutions $j$ of 
\begin{align}\label{ao38bis}
\nabla\cdot j=n_{(0,R)^2}-\lambda\;\;\mbox{in}\;(0,R)^2,\quad
\nu\cdot j=0\;\;\mbox{on}\;\partial(0,R)^2.
\end{align}
Inequality (\ref{ao41}) can be easily derived from the Eulerian description
of optimal transportation (see for instance \cite[Proposition 1.1]{BenBre00} or \cite[Theorem 8.1]{Viltop}). Indeed, the couple $\bra{\rho, j}$, with $\rho = t n_{(0,R)^2} + \bra{1-t} \lambda \ge \frac12$ and $j$ satisfying $\eqref{ao38bis}$ is an admissible candidate for the Eulerian formulation of $W_{(0,R)^2}^2(\lambda,n_{(0,R)^2})$.

\medskip

We now construct a $j$:
For a dyadic square $Q$ with its four children $Q'$, we (distributionally) solve the Poisson equation with piecewise 
constant r.~h.~s.~and no-flux boundary data
\begin{align}\label{ao36}
-\Delta\phi_Q=n_Q-n_{Q'}\;\;\mbox{in}\;Q',\quad
\nu\cdot\nabla\phi_Q=0\;\;\mbox{on}\;\partial Q.
\end{align}
Note that (\ref{ao36}) admits a solution since the integral of the r.~h.~s.~over
$Q$ vanishes by definition of $\{n_Q\}$.
Since the no-flux boundary condition allows for concatenation of $-\nabla\phi_Q$
without creating singular contributions to the divergence, 
\begin{align}\label{ao39}
j(x):=-\sum_{Q\ni x,R \ge r_Q\ge 2r_{Q_*}\;\mbox{for}\; Q_* \ni x}\nabla\phi_Q(x)
\end{align}
which means that the sum\footnote{with the understanding that $j(x)=0$ when the sum is empty.} extends over all dyadic cubes $Q \subset \bra{0,R}^2$ containing $x$ and being strictly coarser than $\cur{Q_*}$, defines indeed a distributional solution of (\ref{ao38bis}).

\medskip

% The crucial advantage of stopping becomes apparent now: Its definition (\ref{ao32}) 
% ensures that 
% the modulus of the r.~h.~s.~of the Poisson equation (\ref{ao36}) is bounded by $\frac{5}{2}$,
% so that the 
The Poincar\'e inequality gives the universal bound
\begin{align*}
\int_Q dx|\frac{1}{r_Q}\nabla\phi_Q|^2\lesssim|Q|\sum_{Q'\;\mbox{child of}\;Q}(n_{Q'}-n_Q)^2.
\end{align*}
Appealing to the trivial estimate $\sum_{Q'\mbox{child of}Q}(n_{Q'}-n_Q)^2$
$\lesssim$ $\sum_{Q'\;\mbox{child of}\;Q}$ $(n_{Q'}-1)^2$, and to the standard
estimate on the variance of the Poisson process, namely
$\mathbb{E}(n_Q-1)^2\lesssim\frac{1}{|Q|}$, we obtain
\begin{align}\label{ao40}
\mathbb{E}\int_Q|\nabla\phi_Q|^2\lesssim |Q|.
\end{align}

\medskip

Now comes the crucial observation: We momentarily fix a point $x$ and 
consider all squares $Q\ni x$. We note that $\nabla\phi_Q(x)$
depends on the Poisson point process only through $\{n_{Q'}\}$ where $Q'$ runs
through the children of $Q$. 
Moreover, since the expectation of the r.~h.~s.~of the Poisson equation
(\ref{ao36}) vanishes, also the expectation of $\nabla\phi_Q(x)$ vanishes.
Hence the sum
\begin{align*}
\sum_{Q\ni x,R \ge r_Q\ge r}\nabla\phi_Q(x)
\end{align*}
is a martingale in the scale parameter $r=R,\frac{R}{2},\frac{R}{4},\cdots$
w.~r.~t.~to the filtration generated by the $\{n_Q\}_{Q}$.
Since $r_* \geq 1$, we thus obtain by the Burkholder inequality \cite[Theorem 6.3.6]{Stroock}
%
%\begin{align*}
%\mathbb{E}\big|\sum_{Q\ni x,r_Q\ge \frac{r_*(x)}{2}}\nabla\phi_Q(x)\big|^2
%\lesssim 
%\mathbb{E}\sum_{Q\ni x,r_Q\ge \frac{r_*(x)}{2}}|\nabla\phi_Q(x)|^2.
%\end{align*}
%
%
\begin{align*}
\mathbb{E}\big|\sum_{Q\ni x,R \ge r_Q\ge 2r_{Q_*}\;\mbox{for}\; Q_* \ni x}\nabla\phi_Q(x)\big|^2
& \leq
\mathbb{E}\sup_{r \ge 1}\big|\sum_{Q\ni x,R \ge r_Q\ge r}\nabla\phi_Q(x)\big|^2 \\
& \lesssim 
\mathbb{E}\sum_{Q\ni x,R \ge r_Q\ge 1}|\nabla\phi_Q(x)|^2.
\end{align*}
Inserting definition (\ref{ao39}) into the l.~h.~s., using triangle inequality, integrating over $x\in(0,R)^2$,
and using (\ref{ao42}) on the r.~h.~s.~we obtain
\begin{align*}
\mathbb{E}\int_{(0,R)^2}|j|^2
\lesssim 
\sum_{Q,R \ge r_Q\ge 1}\mathbb{E}\int_Q|\nabla\phi_Q|^2. 
\end{align*}
Using (\ref{ao41}) on the l.~h.~s.~and (\ref{ao40}) on the r.~h.~s.~yields (\ref{ao38}).
\end{proof} 

\begin{remark}
The same argument in $d>2$ yields the bound $\EE W_{(0,R)^d}\bra{\mu,n}\les 1$. However, the interesting (well known) information from the proof is that in $d>2$ the main contribution comes from the term $W_{(0,R)^d}^2(\mu,\lambda)$ which collects the contributions on the small scales.
\end{remark}

\begin{proof}[Proof of Lemma \ref{lem:upperboundpathwise}.]
W.l.o.g.~it suffices to prove that for the Poisson point process $\mu$ there exists a random radius $r_*$ such that for any  dyadic radii $R \geq r_*$
\begin{equation}\label{eq:uppwlog}
\frac1{R^d} W^2_{(-R,R)^d} \bra{\mu, n} + \frac{R^2}{n} \bra{n-1}^2 \lesssim \ln R. 
\end{equation}
The statement will follow by choosing the maximum of this random radius and the one pertaining to $\nu$. 

\medskip

{\sc Step 1}. We claim that there exists a constant $C$ and a random radius $r_* < \infty$ a.~s. such that for any dyadic radii $R \geq r_*$
\begin{equation}\label{eq:upperWaspath}
\frac1{R^2} W_{(-R,R)^2}^2(\mu,n)\leq C \ln R.
\end{equation}
By Lemma \ref{Lem2} we may assume that \eqref{eq:upperboundmean} holds with $(0,R)^2$ replaced by $(0, 2R)^2$. Furthermore by stationarity of the Poisson point process we may assume that \eqref{eq:upperboundmean} holds in the form
\begin{equation}\label{eq:upperboundmean-RRbox}
\frac1{R^2} \EE W_{(-R,R)^2}^2(\mu,n)\lesssim \ln R,
\end{equation}
provided that $R \gg 1$. By \cite[Proposition 2.7]{GHO1} it follows that for any $\eps>0$, 
\[
\mathbb{P} \bra{ \frac{1}{R^2 \ln R}\abs{W_{(-R,R)^2}^2 \bra{\mu, n}- \mathbb{E} W_{(-R,R)^2}^2 \bra{\mu , n}} > \eps} \lesssim  \exp \bra{- C \eps \ln R}.
\]
Then \eqref{eq:upperWaspath} follows by arbitrarity of $\eps$ together with a Borel-Cantelli argument.

\medskip

{\sc Step 2}.
We show that there exists a constant $C$ and random radius $r_* < \infty$ a.~s. such that for any dyadic radii $R \geq r_*$
\begin{equation}\label{eq:ndatatermupper}
\frac{R^2}{n} \bra{n-1}^2 \leq C \ln R.
\end{equation}
Since for large $R \gg 1$ we have $\frac{\ln R}{R^2} \ll 1$ \eqref{eq:ndatatermupper} is equivalent to
\[
{R^2} \bra{n-1}^2 \lesssim \ln R.
\] 
Since $n 4 R^2$ is Poisson distributed with parameter $4 R^2$ by Cram\'er-Chernoff's bounds \cite[Theorem 1]{BouLuMa}
\[
\begin{split}
\P \bra{R^2 \bra{n-1}^2 > \ln R} & = \P \bra{\abs{n 4 R^2-  4 R^2} > C R \ln^\frac12 R}  \lesssim \exp \bra{- C \ln R}.
\end{split}
\]
Finally, by a Borel-Cantelli argument \eqref{eq:ndatatermupper} holds for any dyadic radii $R \geq r_*$.
\end{proof}

 \bibliographystyle{abbrv}

\bibliography{OT}

\begin{thebibliography}{10}

\bibitem{AKT84}
M.~{Ajtai}, J.~{Koml\'os}, and G.~{Tusn\'ady}.
\newblock {On optimal matchings.}
\newblock {\em {Combinatorica}}, 4:259--264, 1984.

\bibitem{AGS19}
L.~Ambrosio, F.~Glaudo, and D.~Trevisan.
\newblock On the optimal map in the 2-dimensional random matching problem.
\newblock {\em Discrete Contin. Dyn. Syst.}, 39(12):7291--7308, 2019.

\bibitem{AmStTr16}
L.~Ambrosio, F.~Stra, and D.~Trevisan.
\newblock A {PDE} approach to a 2-dimensional matching problem.
\newblock {\em {Probab. Theory Relat. Fields}}, 173(1-2):433--477, 2019.

\bibitem{BenBre00}
J.-D. Benamou and Y.~Brenier.
\newblock A computational fluid mechanics solution to the monge-kantorovich
  mass transfer problem.
\newblock {\em Numerische Mathematik}, 84(3):375--393, Jan 2000.

\bibitem{BobLe}
S.~Bobkov and M.~Ledoux.
\newblock A simple {F}ourier analytic proof of the {AKT} optimal matching
  theorem, 2019.

\bibitem{BouLuMa}
S.~Boucheron, G.~Lugosi, and P.~Massart.
\newblock {\em Concentration inequalities}.
\newblock Oxford University Press, Oxford, 2013.

\bibitem{CaLuPaSi14}
S.~Caracciolo, C.~Lucibello, G.~Parisi, and G.~Sicuro.
\newblock Scaling hypothesis for the {E}uclidean bipartite matching problem.
\newblock {\em Physical Review E}, 90(1), 2014.

\bibitem{ChPePeRo10}
S.~Chatterjee, R.~Peled, Y.~Peres, and D.~Romik.
\newblock Gravitational allocation to poisson points.
\newblock {\em Annals of mathematics}, pages 617--671, 2010.

\bibitem{GHO1}
M.~{Goldman}, M.~{Huesmann}, and F.~{Otto}.
\newblock {A large-scale regularity theory for the Monge-Ampere equation with
  rough data and application to the optimal matching problem}.
\newblock {\em arXiv:1808.09250}, Aug 2018.

\bibitem{GHO}
M.~{Goldman}, M.~{Huesmann}, and F.~{Otto}.
\newblock Quantitative linearization results for the {M}onge-{A}mp\`ere
  equation.
\newblock {\em Com. Pure and Appl. Math}, 2021.

\bibitem{GolTre}
M.~Goldman and D.~Trevisan.
\newblock Convergence of asymptotic costs for random {E}uclidean matching
  problems, 2020.

\bibitem{HoHoPe06}
C.~Hoffman, A.~E. Holroyd, and Y.~Peres.
\newblock A stable marriage of poisson and lebesgue.
\newblock {\em The Annals of Probability}, 34(4):1241--1272, 2006.

\bibitem{Ho11}
A.~E. Holroyd.
\newblock Geometric properties of poisson matchings.
\newblock {\em Probability theory and related fields}, 150(3):511--527, 2011.

\bibitem{HoJaWa20}
A.~E. Holroyd, S.~Janson, and J.~W{\"a}stlund.
\newblock Minimal matchings of point processes.
\newblock {\em arXiv preprint arXiv:2012.07129}, 2020.

\bibitem{HoPePeSc09}
A.~E. Holroyd, R.~Pemantle, Y.~Peres, and O.~Schramm.
\newblock Poisson matching.
\newblock {\em Annales de l'I.H.P. Probabilit\'es et statistiques},
  45(1):266--287, 2009.

\bibitem{Hu16}
M.~Huesmann.
\newblock {Optimal transport between random measures}.
\newblock {\em Annales de l'Institut Henri Poincar\'e, Probabilit\'es et
  Statistiques}, 52(1):196 -- 232, 2016.

\bibitem{HuSt13}
M.~{Huesmann} and K.-T. {Sturm}.
\newblock {Optimal transport from Lebesgue to Poisson}.
\newblock {\em {Ann. Probab.}}, 41(4):2426--2478, 2013.

\bibitem{kallenberg}
O.~Kallenberg.
\newblock {\em Foundations of modern probability}.
\newblock Probability and its Applications (New York). Springer-Verlag, New
  York, second edition, 2002.

\bibitem{lastpenrose}
G.~Last and M.~Penrose.
\newblock {\em Lectures on the Poisson Process}.
\newblock Institute of Mathematical Statistics Textbooks. Cambridge University
  Press, 2017.

\bibitem{Le17}
M.~Ledoux.
\newblock On optimal matching of {G}aussian samples.
\newblock {\em Zap. Nauchn. Sem. S.-Peterburg. Otdel. Mat. Inst. Steklov.
  (POMI)}, 457(Veroyatnost' \ i Statistika. 25):226--264, 2017.

\bibitem{MaTi16}
R.~Mark{\'o} and {\'A}.~Tim{\'a}r.
\newblock A poisson allocation of optimal tail.
\newblock {\em The Annals of Probability}, 44(2):1285--1307, 2016.

\bibitem{Stroock}
D.~W. Stroock.
\newblock {\em Probability Theory: An Analytic View}.
\newblock Cambridge University Press, 2 edition, 2010.

\bibitem{Ta94}
M.~{Talagrand}.
\newblock {The transportation cost from the uniform measure to the empirical
  measure in dimension $\geq 3$.}
\newblock {\em {Ann. Probab.}}, 22(2):919--959, 1994.

\bibitem{Viltop}
C.~Villani.
\newblock {\em Topics in optimal transportation}, volume~58 of {\em Graduate
  Studies in Mathematics}.
\newblock American Mathematical Society, Providence, RI, 2003.

\bibitem{Wu00}
L.~Wu.
\newblock A new modified logarithmic {S}obolev inequality for {P}oisson point
  processes and several applications.
\newblock {\em Probab. Theory Related Fields}, 118(3):427--438, 2000.

\end{thebibliography}
\end{document}